\documentclass[11pt]{article}
\usepackage{amssymb,latexsym,amsmath, amsthm}
\usepackage{amsfonts,amssymb, latexsym, amsmath, amsthm,mathrsfs, verbatim,geometry}

\usepackage{amsfonts, graphics}
\usepackage{mathtools}
\usepackage{hyperref} 
\usepackage{csquotes} 
\usepackage{enumitem,moreenum} 
\usepackage{float}
\usepackage{dsfont} 
\numberwithin{equation}{section}
\def\R{\mathbb R}

\def\N{\mathbb N}
\def\Z{\mathbb Z}

\hypersetup{
    colorlinks,
    citecolor=red,
    filecolor=green,
    linkcolor=blue,
    urlcolor=black
}

\def\supp{\mathrm{supp}\,}

\def\T{{\mathcal T}}
\def\P{{\mathcal P}}

\def\A{{\mathcal A}}

\def\X{{\mathcal X}}
\def\Ri{{\mathcal R}}

\def\pa{\partial}
\def\l{\lambda}

\def\M{\mathcal M}

\def\H{\mathcal H}
\def\lv{\left\vert}
\def\rv{\right\vert}
\def\k{\varepsilon}
\def\pkl{p_{m,\ell}}

\def\Ldk{L_\delta^m}

\def\P{\mathcal{P}}

\renewcommand{\epsilon}{\varepsilon}
\def\barI{\bar{I}}
\def\Zast{{\mathbb Z^{\ast}}}

\let\originalleft\left
\let\originalright\right
\renewcommand{\left}{\mathopen{}\mathclose\bgroup\originalleft}
\renewcommand{\right}{\aftergroup\egroup\originalright}

\newcommand*\diff{\mathop{}\!\mathrm{d}} 

\newtheorem{theorem}{Theorem}[section]
\newtheorem{definition}[theorem]{Definition}
\newtheorem{prop}[theorem]{Proposition}

\newtheorem{corollary}[theorem]{Corollary}

\newtheorem{lemma}[theorem]{Lemma}

\newtheorem{remark}[theorem]{Remark}

\usepackage{graphicx}
\usepackage{caption}
\usepackage{tikz}
\usetikzlibrary{calc}
\usetikzlibrary{arrows.meta}

\usepackage{color}
\def\bcr{\begin{color}{red}}
\def\bcb{\begin{color}{blue}}
\def\ec{\end{color}}
\def\be{\begin{equation}}
\def\ee{\end{equation}}
\def\Cr{C_{\text{res}}}

\def\supp{\mathrm{supp}\,}

\def\T{{\mathcal T}}
\def\L{{\mathcal L}}
\def\P{{\mathcal P}}

\def\A{{\mathcal A}}

\def\Ri{{\mathcal R}}

\def\pa{\partial}
\def\l{\lambda}

\def\M{\mathcal M}

\def\H{\mathcal H}

\def\Lf{\tilde{\mathcal L}}
\def\fke{f^{k,\k}}

\newcommand{\Rmin}{R_{\mathrm{min}}}
\newcommand{\Rmax}{R_{\mathrm{max}}}

\newcommand{\AEL}{\mathbb{A}}

\newcommand{\Emin}{E_{\mathrm{min}}}
\newcommand{\rL}{r_{\ast}}
\newcommand{\im}{\mathrm{im}}
\newcommand{\sigmaess}{\sigma_{\mathrm{ess}}}
\newcommand{\Tmax}{T_{\mathrm{max}}}
\newcommand{\Tmin}{T_{\mathrm{min}}}
\newcommand{\Tpmax}{T_{\mathrm{max}}'}
\newcommand{\Tpmin}{T_{\mathrm{min}}'}
\newcommand{\Tppmax}{T_{\mathrm{max}}''}
\newcommand{\Tppmin}{T_{\mathrm{min}}''}
\newcommand{\Ms}{M_{\mathrm{s}}}

\begin{document}

\title{Damping versus oscillations for a gravitational Vlasov-Poisson system}
\author{M.~Had\v zi\'c\thanks{University College London, UK. Email: m.hadzic@ucl.ac.uk}, G.~Rein\thanks{University of Bayreuth, Germany. Email: gerhard.rein@uni-bayreuth.de}, M.~Schrecker\thanks{University College London, UK. Email: m.schrecker@ucl.ac.uk}, C.~Straub\thanks{University of Bayreuth, Germany. Email: christopher.straub@uni-bayreuth.de}}

\maketitle

\begin{abstract}
We consider a family of isolated inhomogeneous steady states to the gravitational Vlasov-Poisson system with a point mass at the centre. They are parametrised by the polytropic index $k>1/2$, 
so that the phase space density of the steady state is $C^1$ at the vacuum boundary if and only if $k>1$. 
We prove the following sharp dichotomy result: if $k>1$ the linear perturbations Landau damp and if $1/2< k\le1$ they do not.

The above dichotomy is a new phenomenon and highlights the importance of steady state regularity at the vacuum boundary in the discussion of long-time behaviour of the perturbations. Our proof of (nonquantitative) gravitational relaxation around steady states with $k>1$ is the first such result for the gravitational Vlasov-Poisson system. 
The key step in the proof is to show that no embedded eigenvalues exist in the essential spectrum of the linearised system.

\end{abstract}

\tableofcontents

\section{Introduction}

The problem of relaxation of stellar systems is a central question in the study of the dynamics of galaxies. It was explored in the pioneering works of Lynden-Bell~\cite{LB1962,LB1967} in the 1960's, who was the first to point out an intimate connection between
galaxy relaxation and the validity of so-called gravitational Landau damping. Landau damping originally referred to a well-known equilibration mechanism for the linearised electrostatic Vlasov-Poisson system around spatially
homogeneous steady states discovered in 1946~\cite{Landau1946}.  In the gravitational case, the term Landau damping was used in~\cite{LB1962} (see also~\cite{BiTr} for an exhaustive list of references to the physics literature) to refer to the decay of macroscopic quantities of the linearised perturbations about a given steady state. 

To study the stability around isolated and localised self-gravitating galaxies, one is forced to consider spatially
inhomogeneous densities and this considerably complicates the stability analysis. There is a continuum of steady states of the gravitational Vlasov-Poisson (VP) system whose infinite-dimensional character is related to the invariance of the VP-system under the action of measure preserving diffeomorphisms. Moreover, the relevant steady states are compactly supported in
both the space and the velocity variable, which means that particles are trapped in a finite region of phase-space, and this can act as an obstruction to decay.

In this work we construct a family of steady states
for which we show that the question of relaxation depends strongly  
on the regularity of the equilibrium at the vacuum boundary. If the steady state is below a certain regularity threshold we prove that the linearised operator has pure oscillations in its spectrum and no damping occurs. If, by contrast, the steady state is above the threshold, there is no pure point spectrum and one can prove non-quantitative decay results using the RAGE theorem. This dichotomy is a striking feature of the gravitational dynamics, and we believe
the methods developed in this paper to have a wide range of applicability. 

To focus on the main ideas,
we consider the radial gravitational Vlasov-Poisson system including
a fixed central potential generated by a point mass of size $M>0$, and we assume that all the particles have angular momentum of fixed modulus.\footnote{We note that the situation of steady states with fixed modulus of angular momentum is discussed in~\cite[Sc.~3.1]{RiSa2020} and in the plasma case see also~\cite{PaWi2021}. 
}
This symmetry reduction removes several technical difficulties and allows us to focus on the key new ideas. The case of general radial equilibria will be addressed in~\cite{HaReScSt2}.
The system reads
\begin{equation}
  \label{eq:VlasovVfixedL}
  \partial_tf+w\,\partial_rf-\left(U'+\frac M{r^2}-\frac {L}{r^3}\right)\partial_wf=0, \\
\end{equation}
\begin{equation}
  \label{E:POISSON}
  U' = \frac{4\pi}{r^2}\int_0^rs^2\rho(t,s)\diff s,\qquad \lim\limits_{r\to\infty}U(t,r)=0, \\
\end{equation}
\begin{equation}
  \label{eq:rhoVfixedL}
  \rho(t,r)=\frac\pi{r^2}\int_{\R}f(t,r,w)\diff w.
\end{equation}
Here $f(t,r,w)\ge0$ is the phase-space number density, a function of time $t\in \R$,
radial position $r>0$, and radial velocity $w\in \R$,
$U(t,r)$ is the gravitational potential induced by the stars of the galaxy,
and $\rho(t,r)\ge0$ their macroscopic mass density. 
The system~\eqref{eq:VlasovVfixedL}--\eqref{eq:rhoVfixedL} is the radial VP-system for an ensemble of particles all of which have angular momentum
with the same squared modulus $L>0$.  

We consider a class of steady states to~\eqref{eq:VlasovVfixedL}--\eqref{eq:rhoVfixedL} of the form
\begin{equation}\label{eq:ansatzfeps}
	\fke(r,w)=\varphi(E(r,w))=\k\,\tilde\varphi(E(r,w)), \qquad  \tilde\varphi(E)=\left(E_0-E\right)_+^k,
\end{equation}
where $\k>0$ is a size-parameter and $k>\frac12$ the polytropic exponent. Here 
\begin{align}
	E(r,w)&=\frac12\,w^2+\Psi(r), \label{eq:defE} \\
	\Psi(r)&=U(r)-\frac Mr+\frac L{2r^2} \label{E:EFFPOT}
\end{align}
are the particle energy and the effective potential respectively, while 
the cut-off energy $E_0<0$ is implicitly determined through the equation satisfied by the steady state.
The gravitational potential~$U$ is induced by~$\fke$ through~\eqref{E:POISSON}--\eqref{eq:rhoVfixedL}.
For completeness of exposition, the  existence of such steady states
with finite radius and finite mass
are shown in Section~\ref{S:STEADYSTATES}; this is actually easier than in
the usual situation without a central point mass, cf.~\cite{RaRe13}. More precisely, fix a $k>\frac12$. Then for any $\k>0$  there exists
a whole 1-parameter family of steady states of the form~\eqref{eq:ansatzfeps} parametrised by the parameter 
\be\label{E:KAPPADEF}
\kappa\coloneqq E_0-U(0)<0
\ee 
which has the meaning of 
a relative gravitational potential at the origin. 
The resulting phase-space support is compact and the associated macroscopic density $\rho(r)$ is of size $\mathcal O_{\k\to0}(\k)$, supported 
on a compact spherical shell $[\Rmin,\Rmax]$ of thickness $\mathcal O_{\k\to0}(1)$ with a delta distribution of mass~$M$ centred at the origin, see Figure~\ref{F:SS}. 
The parameter~$\kappa$ determines the inner vacuum radius $\Rmin>0$ as well as the (finite) limit of the outer vacuum radius $\Rmax$ as $\epsilon\to0$.
We shall suppress the dependence on~$\kappa$ and fix it to any value satisfying the {\em single gap condition} 
 \begin{equation}\label{E:SINGLEGAP}
	-2^{-\frac23}\,\frac{M^2}{2L}<\kappa<0.
\end{equation}
As shown in Corollary~\ref{C:NOBANDS}, condition~\eqref{E:SINGLEGAP} ensures that the essential spectrum of the linearised operator is simply connected for the relevant equilibria.
The pivotal question is the dependence of the stability behaviour of the steady states $f^{k,\epsilon}$ on the parameters $k$ and $\epsilon$.

\begin{figure}
	\begin{center}
		\begin{tikzpicture}[>=stealth]
			\def\wx{3}      
			\def\wxx{1.7}  
			\def\wy{1.5}    
			\def\wz{0.75}   

			\draw[thick,fill=gray!50] (-3,0) -- (3,0) arc(0:-180:3) --cycle;

			\draw[blue, fill=gray] (0,0) ellipse (\wx cm and \wy cm);

			\draw[dashed,thick] (-\wxx,0) -- (\wxx,0) arc(0:-180:\wxx) --cycle;

                         \begin{scope}
			\draw[blue,fill=white] (0, 0) ellipse (\wxx cm and \wz cm);
			\end{scope}
			
			\draw[thick] (0,0) -- (0,2.4);
			\draw[fill=black] (0,0) circle (1pt);
			\node at (0, -0.2) {$r=0$};
			\node at (0.3,2.6) {M$\delta_0$};

			\coordinate (O) at (6,-1.5){};
			
			\draw[thick] (O)--(11,-1.5);
			
			\draw[thick] (O)--(6,1);
			
			\coordinate[label=below:$\Rmin$] (K) at (7.0,-1.5);
			
			\coordinate [label=below:$\Rmax$] (L) at (9.9,-1.5);
			
			\coordinate (M) at (8.5,-0.5);
			
			\coordinate (P) at (8.5,-1.0);

			\draw[thick] (K) .. controls +(1.0,0.1) and  +(-0.5,0)  .. (P);
			
			\draw[thick] (P) .. controls +(0.5,0) and  +(-1,0.1)  .. (L);

			\node at (M) {$\rho(r)$};
			
			\node at (11,-1.7) {$r$};
			
			\draw[fill=black] (K) circle (1pt);
			
			\draw[fill=black] (L) circle (1pt);
			
		\end{tikzpicture}
	\end{center}
	\caption{Schematic depiction of the lower hemisphere of the spherical shell (on the left) and the macroscopic density distribution $\rho(r)$ (on the right).}
	\label{F:SS}
\end{figure}
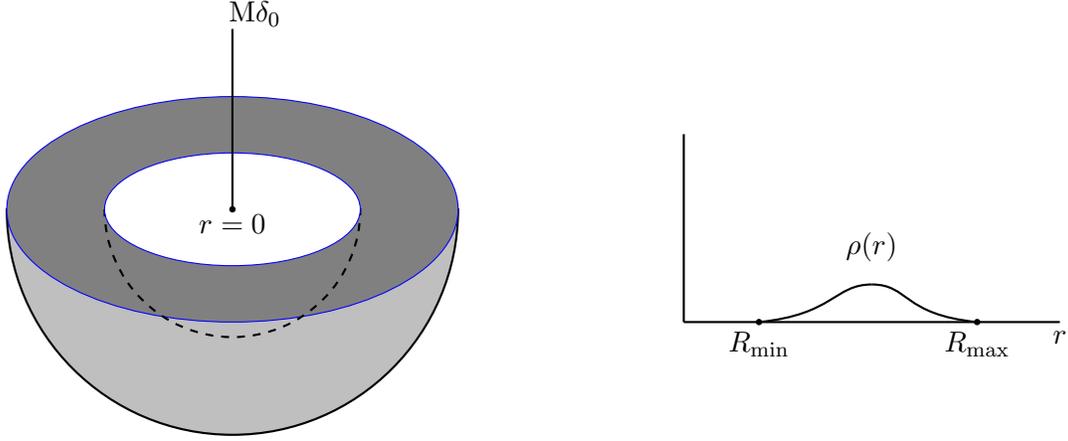


We linearise the system~\eqref{eq:VlasovVfixedL}--\eqref{eq:rhoVfixedL} around a fixed steady state~$\fke$.
If we denote the linear perturbation by $F$, a  straightforward calculation gives the linearisation
\begin{equation}\label{E:LVP}
\pa_t F + \Lf F = 0,
\end{equation}
where
\begin{equation}\label{E:FOOP}
	\Lf F\coloneqq\T \left(F + \lv\varphi'(E)\rv\,U_F\right), 
\end{equation}
the {\em transport operator $\T$} is given by
\begin{equation}\label{eq:transportdef}
	\T\coloneqq w\,\partial_r-\Psi'(r)\,\partial_w,
\end{equation}
and $U_F$ solves the radial Poisson equation
\begin{equation}\label{eq:Ufprime}
  U_F'(r)=\frac{4\pi}{r^2}\int_0^rs^2\,\rho_F(s)\diff s
  =\frac{4\pi^2}{r^2}\int_0^r\int_{\R}F(s,w)\diff w\diff s,\qquad\lim_{r\to\infty}U_F(r)=0.
\end{equation}
Alternatively, one can apply the classical Antonov trick~\cite{An1961} and split~\eqref{E:LVP} into separate equations for the even and odd in $w$ parts $f_\pm(r,w)=\frac12(F(r,w)\pm F(r,-w))$ of the perturbation~$F$. The linear evolution is then fully described by the following second order system for $f_-$: 
\begin{equation}\label{E:LVP2}
	\partial_t^2f_-+\L f_-=0.
\end{equation}
The {\em linearised operator} (also referred to as the {\em Antonov operator}) takes the form
\begin{equation}\label{eq:Ldef}
	\L\coloneqq-\T^2-\Ri,
\end{equation}
where 
the {\em gravitational response operator~$\Ri$} is given by
\begin{equation}\label{eq:responsedef}
	\Ri g\coloneqq4\pi^2\,\lv \varphi'(E)\rv \,\frac w{r^2}\,\int_{\R}\tilde w\,g(r,\tilde w)\diff\tilde w.
\end{equation}
Functional-analytic properties of the operators $\tilde{\L}$ and $\L$ are discussed in
Section~\ref{S:LINEARISATION}. We shall mostly work with the second order formulation~\eqref{E:LVP2}, although the analysis can be  carried out analogously in the first order formulation~\eqref{E:LVP}. 
The natural Hilbert space for our analysis is the weighted $L^2$-space
\begin{equation*}
	H\coloneqq\left\{ f\colon\Omega\to\R\mid f\text{ measurable and }\|f\|_{H}<\infty\right\},
\end{equation*}
where $\|\cdot\|_H$ is induced by the inner product
\begin{equation}\label{eq:normHdef}
	\langle f,g\rangle_H\coloneqq\int_{\Omega}\frac1{|\varphi'(E)|}\,f(r,w)\,g(r,w)\diff(r,w)
\end{equation}
and $\Omega=\{\fke>0\}$ is the interior of the steady state support. Note that the integrand in~\eqref{eq:normHdef} is well-defined since 
\be\label{E:ASS}
\varphi'(E(r,w))<0, \qquad (r,w)\in\Omega. 
\ee 
Since~$\L$ only covers the evolution of the odd-in-$w$ part of the linear perturbation, we further define the subspace of~$H$ consisting of odd-in-$w$ functions as
\begin{equation*}
	\H\coloneqq\{f\in H\mid f \text{ is odd in }w\text{ a.e.\ on }\Omega\}.
\end{equation*}
We shall see in Section~\ref{S:LINEARISATION} that~$\L$ is self-adjoint on~$\H$ when defined on the proper domain~$\mathrm D(\L)$.

The monotonicity condition~\eqref{E:ASS} is known as the Antonov linearised
stability criterion. For the case without a central point mass
it was shown in the physics literature~\cite{DoFeBa,KS} that it implies 
the spectral stability, which is equivalent to the non-negativity 
of the quadratic form $\langle\L h,h\rangle_H$ on $\mathrm D(\L)$. This result can also be thought of as the analogue of the Penrose stability criterion for plasmas~\cite{MoVi2011}. Moreover, by a simple modification of the arguments in~\cite{GuRe2007,LeMeRa11,LeMeRa12} one can prove that the steady states under consideration are nonlinearly orbitally stable in our symmetry class, which is essentially due to the energy subcritical nature of the problem. By contrast, nothing is known about the asymptotic-in-time behaviour of solutions close to such steady states and,
unlike the classical Landau damping for plasmas, it is a priori unclear whether any form of
damping occurs for the linearised dynamics~\eqref{E:LVP2}.


To provide a meaningful formulation of Landau damping, we must consider initial data in the complement of the kernel of the operator $\L$. Viewed as an operator on $\H$ the kernel of $\L$ is trivial, see Lemma~\ref{lem:Lproperties}.


\begin{definition}[Nonquantitative Landau damping]\label{D:LANDAUDAMPING}
  For $k>\frac12$ and $\k>0$, let $\fke$ denote the steady state
  of the Vlasov-Poisson system~\eqref{eq:VlasovVfixedL}--\eqref{eq:rhoVfixedL}
  of the form~\eqref{eq:ansatzfeps}. We say that the linearised Vlasov-Poisson
  equation~\eqref{E:LVP2} {\em Landau damps}, if for any initial data
  $f_0\in \mathrm D(\L)\subset\H$,
  \begin{align}\label{E:LANDAU}
    \lim_{T\to\infty}\frac1T\int_0^T\|\nabla U_{\T f(t,\cdot)}\|_{L^2(\mathbb R^3)}^2\diff t = 0,
  \end{align}
  where $\mathbb R_+\ni t\to f(t,\cdot)\in \H$ is the unique solution to~\eqref{E:LVP2} with initial data $f(0,\cdot)=f_0$.
\end{definition}


Definition~\ref{D:LANDAUDAMPING} connects to the first-order dynamics as follows.
If $t\mapsto F(t,\cdot)$ solves~\eqref{E:LVP}, then
$\pa_tU_F = U_{\pa_t F}= U_{\pa_t f_+} = - U_{\T f_-} = - U_{\T f}$, where we recall $f_+$ is the even part of $F$, $f_-=f$ is the odd part.
It follows that~\eqref{E:LANDAU}
is equivalent to the claim
\[
\lim_{T\to\infty}
\frac1T\int_0^T\|\nabla \pa_tU_{F(t,\cdot)}\|_{L^2(\mathbb R^3)}^2\diff t = 0.
\]


Formula~\eqref{E:LANDAU} implies a very weak form of decay of the macroscopic quantity $\|\nabla U_{\T f(t,\cdot)}\|_{L^2(\mathbb R^3)}^2$, without a rate. 
As we shall see in Section~\ref{S:RAGE}, this decay
follows from
 the RAGE theorem because no pure point spectrum is present. 
We chose to define Landau damping via~\eqref{E:LANDAU} for specificity, but one could in principle consider other macroscopic quantities.
We state our main theorem. 


\begin{theorem}[Oscillation vs.\ relaxation]\label{T:MAINTHEOREM}
  For $k>\frac12$ and $\k>0$, let $\fke$ denote the steady state of the Vlasov-Poisson
  system~\eqref{eq:VlasovVfixedL}--\eqref{eq:rhoVfixedL} of the form~\eqref{eq:ansatzfeps}.
  Then the following dichotomy holds.
  \begin{enumerate}[label=(\alph*)]
  \item\label{it:main1} For any $\frac12<k\le 1$ there exists an $\k_0=\k_0(k)>0$ such that
    for any $0<\k<\k_0$ the system~\eqref{E:LVP2} {\em does not} damp.
    More precisely, there exists at least one strictly positive eigenvalue of $\L$. 
\item\label{it:main2} For any $k>1$ there exists an $\k_0=\k_0(k)>0$ such that for any $0<\k<\k_0$ the system~\eqref{E:LVP2} {\em does} Landau damp  in the sense of Definition~\ref{D:LANDAUDAMPING}. In particular, the point spectrum of $\L$ is empty.
\end{enumerate}
\end{theorem}


Our aim in the present paper is not to compute the rate of decay in the
damped case ($k>1$ and $0<\k\ll1$),
but instead to focus on the dichotomy stated in Theorem~\ref{T:MAINTHEOREM}. 
An important consequence of the theorem is that the gravitational Landau damping is sensitive to the regularity of the underlying steady state. 
Note that the steady states $\fke$ are always $C^\infty$ in the interior of their phase-space support and $C^{\lfloor k\rfloor, k-\lfloor k\rfloor}$  up to and including the vacuum boundary $\{E=E_0\}$.
Therefore the regularity limitation stems from the boundary behaviour.\footnote{We note that the steady states $\fke$ fall in the regularity class for which one can prove local-in-time well-posedness.} 



  The polytropes defined through~\eqref{eq:ansatzfeps} are very commonly studied  
  in the gravitational kinetic theory~\cite{BiTr}. However, a simple examination of the proof shows that it is only the regularity of $\fke$ near the phase-space vacuum boundary that discriminates between Landau damping and oscillations. We may therefore use more general ansatz functions $\tilde\varphi(E)$ whose Taylor expansion near the vacuum reads $\tilde\varphi(E)\approx (E_0-E)^k + o_{E\to E_0}\left((E_0-E)^k\right)$; here $k$ plays the same role as in Theorem~\ref{T:MAINTHEOREM}. 
  For example, linearised perturbations of the King model
  $\varphi_{\text{King}}(E) = \k (e^{E_0-E}-1)_+$ with $0<\epsilon\ll1$ do not damp.

A further direct consequence of the proof of the main theorem is that our methods can be used to give a criterion for the absence of embedded eigenvalues for general radial steady states, i.e., with and without the central point mass. At a technical level, the presence of the point mass allows us to consider \enquote{small} steady states by introducing the parameter $0<\k\ll1$. This in turn allows us to rigorously verify many of the structural properties of the steady states, most notably the monotonicity of the period function. Such properties are expected to be true for general steady states even when the small parameter is not available, but this is outside the scope of the present work, see~\cite{HaReScSt2}.


The RAGE theorem was used to show nonquantitative damping around certain steady states of the 2D Euler equations by Lin and Zeng, see, e.g.,~\cite[Thm.~11.7]{LiZe2017}.
Our set-up is manifestly based on the second order formulation~\eqref{E:LVP2}. However, once we establish the absence of point spectrum, we can equivalently work in the first order formulation~\eqref{E:LVP}. Following the strategy of~\cite{LiZe2017} we can restrict the dynamics to the invariant subspace of so-called {\em linearly dynamically accessible} perturbations $\im(\T)$ and exhibit weak decay of $\|\nabla U_F\|_{L^2(\R^3)}$ for the data in the orthogonal complement of the kernel of $\tilde\L$.   The subject of quantitative inviscid damping and the nonlinear stability around (typically) shear flow  solutions of 2D Euler has been a very active area in the past decade, following the nonlinear stability result of Bedrossian and Masmoudi~\cite{BeMa2015}. Without attempting to give an exhaustive overview, we refer the reader to the introductions of the recent articles~\cite{IoIyJi2023,MaZh2020}, the review article~\cite{BeGeMa2019}, and the lecture~\cite{IoICM2022} for an exhaustive list of references.


Theorem~\ref{T:MAINTHEOREM} is the first result which
shows that Landau damping occurs around compactly supported,
inhomogeneous equilibria of the gravitational Vlasov-Poisson system. The stated dichotomy between relaxation and oscillation, as well as the sharp transition threshold $k=1$ are, to our knowledge, new. This situation is
reminiscent of the well-known fact in the spectral theory of Schr\"odinger operators $-\Delta+V$ where the smallness of the potential $V$ (in the right sense) helps to exclude bound states in dimension $d=3$ and cannot exclude them when $d=1$. In this analogy, the polytropic index $k$, which measures the regularity of the steady state at the vacuum boundary, plays the role of the dimension~$d$. 

The possibility of oscillatory linear behaviour and the contrast to gravitational damping have been discussed in the physics literature~\cite[Ch.~5]{BiTr}, see also~\cite{Ma,Louis1992,Weinberg1994,BaYa2013,BaOlYa2015}. The observation that the smoothness of the perturbed steady state is relevant for the nonlinear damping 
is clearly stated in the numerical work of Ramming and Rein~\cite{RaRe18}, where
radial steady states without a central point mass are considered. We also point out the
influential work
of Kalnajs~\cite{Kalnajs1971,Kalnajs1977} where a formal approach is developed to study the decay of macroscopic quantities for linear perturbations using action-angle variables, see the discussion in~\cite[Sc.~5.3.2]{BiTr}.

The question of gravitational relaxation was investigated in the pioneering work
of Lynden-Bell~\cite{LB1962,LB1967}, see also~\cite{BiTr,LB1994}, who recognised
that there exists a phase-mixing mechanism which could explain damping around stationary galaxies. By definition, phase mixing refers to a process according to which macroscopic observables, like the spatial density or gravitational potential associated to the solutions of the pure transport problem 
\be\label{E:PURETRANSPORT}
\pa^2_{t}f-\T^2 f = 0,
\ee
decay in time. 
This mechanism was informally described by Lynden-Bell~\cite{LB1962} and relies on the crucial monotonicity assumption $T'(E)\neq0$ on $\bar I$, where $T(E)$ is the particle period function and $\bar I$ is the action interval of the steady state, see~\eqref{eq:Iepsdef}--\eqref{eq:Tepsdef}. Intuitively, this monotonicity condition allows the particles to explore the phase space very efficiently and therefore creates a mixing effect. In practice one can use arguments
\`a la Riemann-Lebesgue lemma~\cite{BrCZMa2022,RiSa2020} or vector field commutators~\cite{ChLu2022,MoRiVa2022}
to obtain decay.
However, equation~\eqref{E:PURETRANSPORT} is not the linearised dynamics around the steady state, and Theorem~\ref{T:MAINTHEOREM} shows that no mixing occurs when $\frac12<k\le1$ despite the fact that the pure transport part does mix irrespective of how small the gravitational response term $\mathcal R$ is. More precisely, we show that the collective response of the gravitational system as measured by the operator-valued potential~$\Ri$ can create nontrivial pure point spectrum. 
Furthermore, the fact that there is some form of mixing in the regime $k>1$ as implied by Theorem~\ref{T:MAINTHEOREM} is highly nontrivial and involves a careful analysis of the response operator~$\Ri$. 

In the plasma case, nonlinear Landau damping around spatially homogeneous steady states was rigorously shown in the celebrated work of Mouhot and Villani~\cite{MoVi2011}, see also~\cite{BeMaMo2016,GrNgRo2020}. The results of~\cite{MoVi2011} also apply to gravitational interactions (applying the Jeans swindle, see~\cite{BiTr,Ki2003}), but such steady states do not represent isolated solutions of the Vlasov-Poisson system, see also the related work~\cite{Be2021}. For plasma dynamics, linear damping around homogeneous equilibria in the whole space was recently analysed in~\cite{HKNgRo2022,BeMaMo2022}, see  also~\cite{GlSc94}. For a recent nonlinear result see~\cite{IoPaWaWi2022}, for the so-called screened case see~\cite{BeMaMo2018,HKNgRo2021} and for the case of massless electrons see~\cite{GaIa}. Far less is known about damping around spatially
inhomogeneous steady states. Guo and Lin~\cite{GuLi2017} constructed examples of stable BGK waves (that do not contain trapped particles) with a non-empty and with an empty point spectrum. The first Landau damping result for a class of BGK waves with a trapping region was shown by Despr\'es~\cite{Despres2019}. For a recent overview of known results about Landau damping, see~\cite{Be2022}.

The plan of the paper is as follows.
Basic properties of the steady states and the linearised operator are explained in Section~\ref{S:STEADYSTATES}. In Section~\ref{S:ABSENCE} we prove that there are no embedded eigenvalues when $k>1$ and $\k$ is sufficiently small, see Theorem~\ref{T:NOEEV}. In Section~\ref{S:GAP} we derive the criterion for the existence of eigenvalues outside the essential spectrum, see Proposition~\ref{prop:BSMprinciple}. We then use it to show that such eigenvalues exist when $\frac12<k\le1$ and do not when $k>1$, see Theorems~\ref{T:BSYES} and~\ref{T:NOEVINGAP} respectively. Theorem~\ref{T:MAINTHEOREM} is finally proved in Section~\ref{S:RAGE}. In Appendix~\ref{sc:limit} we provide many key results about the underlying family of steady states, most notably various uniform-in-$\k$ bounds for the period function $T(E)$ and its derivatives, as they play a crucial role in our analysis.
Before we enter into the detailed proofs, in Section~\ref{strategy} we give  a short
overview of the general strategy which we employ.

\bigskip 

{\bf Acknowledgments.}
M. Had\v zi\'c's research is supported by the EPSRC Early Career Fellowship EP/S02218X/1.
M. Schrecker's research is supported by the EPSRC Post-doctoral Research Fellowship EP/W001888/1.


\section{An overview of the proof}\label{strategy}

The starting point for our analysis is a reformulation of~\eqref{E:LVP2} in action-angle
variables~\cite{LB1994,BiTr}. 
We denote the minimal particle energy of the steady state by $\Emin$. Letting 
\be\label{eq:Iepsdef}
I\coloneqq]\Emin,E_0[
\ee 
be the \enquote{action} interval, we associate to any $E\in I$ two unique radii $r_-(E)<r_+(E)$ such that $\Psi(r_\pm(E))=E$. Particles are trapped inside the potential well defined
by the effective potential~$\Psi$, and at any fixed energy level $E\in I$,
they oscillate periodically  between their turning points
$r_-(E)$ and $r_+(E)$.
The period $T(E)$ of this motion is  given by the formula
 \begin{equation}\label{eq:Tepsdef}
	T(E)\coloneqq 2\int_{r_-(E)}^{r_+(E)}\frac{\diff r}{\sqrt{2E-2\Psi(r)}},\qquad E\in I.
\end{equation}
The angle $\theta$ parametrises this radial motion,
suitably normalised by the period function.
More precisely, for $(r,w)\in\Omega$ with $w\geq0$ and $E=E(r,w)$ given by~\eqref{eq:defE}, the angle is defined as 
\begin{equation}\label{thetadef}
  \theta(r,w) = \frac1{T(E)}\int_{r_-(E)}^r\frac{\diff s}{\sqrt{2E-2\Psi(s)}}
  \in [0,\frac12].
\end{equation}
Letting $\theta(r,w)=1-\theta(r,-w)$ for $w<0$ leads to the one-to-one change of variables $(r,w)\mapsto(\theta,E)$, where~$\Omega$, i.e., the interior of the support of the steady state in phase space, is mapped onto the cylinder 
\begin{equation*}
	\mathbb S^1\times I.
\end{equation*}
Here $\mathbb S^1$ is the 1-dimensional torus, i.e., $\mathbb S^1 \coloneqq [0,1]$, where~$0$ and~$1$ are identified.
In action-angle variables $(\theta,E)$, the transport operator~$\T$ is now given by the simple formula
\begin{equation*}
\T = \frac1{T(E)}\pa_\theta.
\end{equation*}
This allows us to explicitly determine the essential spectrum of $-\T^2$ in terms of the period function~\eqref{eq:Tepsdef}.
Moreover, the gravitational response operator~$\Ri$ does not affect the essential spectrum
and we obtain that the operator $\L$
has essential spectrum of the
form $[\frac{4\pi^2}{\Tmax^2},\infty[$ for $0<\epsilon\ll1$, where~$\Tmax<\infty$ is
the maximum of the period function~$T$ over $\bar I$,
cf.\ 
Corollary~\ref{C:NOBANDS}. 
Proving these statements mainly relies on
a frequency analysis in the angle variable~$\theta$.
For $f\in L^2(\mathbb S^1)$ we let
\begin{equation}\label{eq:Fouriercoeffdef}
  \hat f(\ell)\coloneqq\int_{\mathbb S^1}f(\theta)\,e^{-2\pi i\ell\theta}\diff\theta,\
  \ell\in\Z;
\end{equation}
Fourier transformations always refer to the variable $\theta$, also for functions
of several variables.

{\bf Absence of embedded eigenvalues (Section~\ref{S:ABSENCE}).}
The hardest part of the proof of Theorem~\ref{T:MAINTHEOREM} is to show that there are no eigenvalues of $\L$ embedded in the essential spectrum when $k>1$, see Theorem~\ref{T:NOEEV}. If we assume, by contradiction, that there exists an eigenvalue of $\L$ of the form $\frac{4\pi^2m^2}{T(E_m)^2}$ for some $(m,E_m)\in\mathbb N\times\bar I$, then $\pm\frac{2\pi i m}{T(E_m)}$ is an eigenvalue of $\tilde\L$, i.e., there exists an $f$ such that $\tilde\L f = \frac{2\pi i m}{T(E_m)} f$. We move to action-angle variables and pass to the Fourier representation 
\begin{align*}
f(\theta,E) =\sum_{\ell\in\mathbb Z} \hat f(\ell,E) e^{2\pi i \ell \theta}, \ \ U_f(\theta,E) = \sum_{\ell\in\mathbb Z} \widehat{U_f}(\ell,E) e^{2\pi i \ell \theta}
\end{align*}
of the unknowns, where 
a simple calculation then shows that for almost every $E\in I$,
\be\label{E:DISPERSIONINTRO}
\hat f(\ell,E) = - T_m  \frac{|\varphi'(E)| \widehat{U_f}(\ell,E)}{T_m -\frac{m}{\ell} T(E)}, \ \ E\in I, \ \ell\in\Zast\coloneqq \Z\setminus\{0\},
\ee
where $T_m:=T(E_m)$, see Lemma~\ref{L:STEP0}.

{\em Gravitational field via the Plancherel identity.}
The key idea is to use the Poisson equation \eqref{eq:Ufprime} to express
$\|\nabla U_f\|_{L^2(\R^3)}^2$ as 
$-16\pi^3 \int f\,U_f T(E)\diff(\theta,E)$. By the
Plancherel identity and~\eqref{E:DISPERSIONINTRO} we then conclude
\begin{equation}\label{E:EEV4INTRO}
\frac1{16\pi^3}\int_{\R^3} |\nabla U_f|^2 \diff x
=  T_m  \sum_{\ell\neq0}\int_I \frac{T(E)|\varphi'(E)|}{T_m -\frac{m}{\ell}T(E)} \lv \widehat{U_f}(\ell,E)\rv^2 \diff E.
\end{equation}
Recall that by~\eqref{eq:ansatzfeps}, $|\varphi'(E)|=\mathcal O(\k)$, so the way to reach
a contradiction is to show that the right-hand side of~\eqref{E:EEV4INTRO}
is bounded by $C\epsilon\int |\nabla U_f|^2\diff x$ and then use the smallness of~$\k$ to absorb $\|\nabla U_f\|_{L^2(\R^3)}$ into the left-hand side. The fundamental difficulty
in doing so are the small denominators appearing inside the integral on the right-hand side of~\eqref{E:EEV4INTRO}.
Clearly there can exist frequency-energy pairs $(\ell,E_\ell)$ such that $T_m-\frac m\ell T(E_\ell)=0$. 

{\em Log-singularity.} The idea is to rewrite such a  possible singularity $\frac1{T_m-\frac m\ell T(E)}$ as $-\frac{\ell}{m T'(E)}\pa_E\left(\log(T_m - \frac m{\ell}T(E))\right)$ in the region where the argument of the logarithm is positive. Note that we are using the property $T'\neq0$ on $\bar I$ in a fundamental way. Our idea is simple; for any frequency $\ell$ we
integrate by parts in $E$ to offload the
$E$-derivative onto the gravitational potential $\widehat{U_f}(\ell,E)$ so that we schematically deal with terms of the form
\be\label{E:GENEXPRESSION}
\k \sum_{\ell\in\mathbb \Zast}
\int_I  g(E)\,|\log(T_m - \frac m{\ell}T(E))|\, |\widehat{U_f}(\ell,E)|\,
|\pa_E\widehat{U_f}(\ell,E)| \diff E,
\ee
where $g$ is some \enquote{well-behaved} weight.
The $\log$-singularity is very mild and the hope is that the integration in $E$ will control it. The small factor of $\k$ is there due to $|\varphi'(E)|\lesssim \k$. 
The first big issue is that the integration-by-parts produces boundary terms, and they must either vanish or have to show up with the correct sign. This is a serious issue, and we must 
carefully analyse the frequency-energy pairs $(\ell,E)$ that produce small and vanishing denominators, see Lemma~\ref{L:RESONANTSIZE0}.
The introduction of the above $\log$-singularity is necessary only at frequencies for which the contributions from the right-hand side of~\eqref{E:DISPERSIONINTRO} are positive. It is a structural feature of the problem  that precisely in this range all the boundary terms are either of {\em good} sign or {\em vanish} due to the regularity and require no further estimates. 
For the vanishing  boundary terms,  we crucially use the regularity assumption $k>1$ which implies $\varphi'(E_0)=0$.

The second key issue is that the minimal point of the effective potential $\Psi$, corresponding to the radius $\rL$ and energy $\Emin$, is a critical point with a strictly positive second derivative. This property, as shown in Lemma~\ref{L:REG}, implies that for any $\theta\in\mathbb S^1$ the map $E\mapsto r(\theta,E)$ is merely $C^{0,\frac12}$ at $E=\Emin$ and in particular
\[
\lv \pa_E r(\theta,E)\rv \lesssim (E-\Emin)^{-\frac12}, \ \ (\theta,E)\in \mathbb S^1\times I,
\]
which creates singular powers of $E-\Emin$ when we try to compare $|\pa_E\widehat{U_f}(\ell,E)|$ to $\pa_r U_f$.
We get around this by introducing positive powers of $E-\Emin$ as weights to
\enquote{de-singularise} $\pa_E\widehat{U_f}(\ell,E)$ and compensate with negative powers of $E-\Emin$ hitting the mild log-singularity, so that we can close the estimates
via Cauchy-Schwarz, see Step~2 of the proof of Theorem~\ref{T:NOEEV}.
The proof shows that the elliptic character of the Poisson equation as manifested through the energy-like identity~\eqref{E:EEV4INTRO} gives the winning strategy, as it permits us to estimate the function $\widehat{U_f}(\ell,E)$ by the derivatives of $U_f$. 
 
{\bf Existence vs.\ absence of eigenvalues in the principal gap (Section~\ref{S:GAP})}.
Existence of positive eigenvalues of~$\L$ below the bottom of the essential spectrum parallels the classical quantum-mechanical problem of finding bound states below the absolutely continuous part of the spectrum of a Schr\"odinger operator. A classical strategy to study bound states is the Briman-Schwinger principle~\cite[Sc.~4.3.1]{LiSe10}, a version of which was pioneered by Mathur~\cite{Ma} for the Vlasov-Poisson system in a different context. In~\cite{HaReSt22,MK} the authors independently derived a criterion for the existence of eigenvalues in the {\em principal gap}
\begin{equation}\label{eq:principalgapdef}
	\mathcal G\coloneqq\left]0,\min\sigmaess(\L)\right[=]0,\frac{4\pi^2}{\Tmax^2}[.
\end{equation}
The work~\cite{HaReSt22} additionally gave examples of steady states where such a criterion can be verified.
We apply a slightly different version of the principle developed in~\cite{HaReSt22} to obtain a necessary and sufficient condition for the existence of eigenvalues in the principal gap $]0,\min\sigmaess(\L)[$, see Proposition~\ref{prop:BSMprinciple}. If $k>1$ this criterion is used in Theorem~\ref{T:NOEVINGAP} to show that there are no eigenvalues in the principal gap and if $\frac12<k\le1$, we use it to prove the
opposite, namely that there are oscillatory eigenvalues in the gap and therefore no damping occurs. 
Both of these proofs are again performed in the $0<\epsilon\ll1$ regime in order to control steady state quantities like the period function~$T$.

{\bf The RAGE theorem and the proof of the main result (Section~\ref{S:RAGE}).}
To complete the proof of Theorem~\ref{T:MAINTHEOREM} we observe that, by the above, the operator $\L$ has empty point spectrum on $\H$ when $k>1$ and $\k>0$ is sufficiently small. We rephrase the linear dynamics $\partial_{t}^2f+\L f =0$ as a first order system and then apply the RAGE theorem~\cite{CyFrKiSi} to show the
nonquantitative decay statement~\eqref{E:LANDAU}. To make this work, we only need to show that the operator
$f\mapsto\lv\varphi'(E)\rv U_{\T f}$
is compact on a suitable function space, 
which again works by virtue of the smoothing properties of the solution operator to the Poisson equation~\eqref{E:POISSON}, see Section~\ref{S:RAGE}.

{\bf Properties of the steady states and the period function $T(E)$ (Appendix~\ref{sc:limit}).}
One of the key analytical tools in our analysis are good uniform-in-$\k$ estimates for steady states $\fke$ with fixed $k>\frac12$ and $0<\k\ll1$.
Most notably, we show that, as $\epsilon\to0$, the period function~$T$ converges in~$C^2$ to the explicitly known
period function $T^0$ generated by the single point mass:
\begin{equation*}
	T^0(E) = \frac\pi{\sqrt2}\,\frac M{(-E)^{\frac32}}.
\end{equation*}
In this way we deduce that~$T=T(E)$ is strictly increasing in~$E$ for $0<\k\ll1$, which is a key ingredient in our analysis. 
In general, (monotonicity) properties of period functions are important in the analysis of the linearised Vlasov-Poisson system, cf.~\cite{HaReSt22,MK}, as well as in the general context of Hamiltonian systems, cf.~\cite{Ch85,ChWa86}.
Further uniform-in-$\epsilon$ bounds on $T$ up to its second derivative ensure that various constants appearing in the proof of Theorem~\ref{T:NOEEV}
are $\k$-independent. 


\section{Steady states and linearisation}\label{S:STEADYSTATES}


\subsection{Existence of steady states}\label{ssc:stst}


\begin{lemma}\label{lem:stst}
Fix the parameter~$\kappa>0$ so that the single-gap condition~\eqref{E:SINGLEGAP} holds. 
Then for any $k>\frac12$ and $\epsilon>0$ there exists a steady state $\fke$ of the system~\eqref{eq:VlasovVfixedL}--\eqref{eq:rhoVfixedL} defined by~\eqref{eq:ansatzfeps}.
The steady state is compactly supported in phase space, more precisely,
\begin{equation}\label{eq:fepssupport}
	\supp(\fke)\subset[\Rmin^0,\Rmax^0]\times[-\frac{\sqrt{2M}}{\sqrt{\Rmin^0}},\frac{\sqrt{2M}}{\sqrt{\Rmin^0}}],
\end{equation}
where $0<\Rmin^0<\Rmax^0<\infty$ are given by
\be\label{eq:Rmindef}
\Rmin^0\coloneqq\frac{-M+\sqrt{M^2+2\kappa L}}{2\kappa} , \qquad \Rmax^0\coloneqq\frac{-M-\sqrt{M^2+2\kappa L}}{2\kappa}.
\ee
The total mass of the steady state is positive and finite, i.e.,
\begin{equation}\label{eq:Mepsdef}
	0<\Ms\coloneqq4\pi\int_0^\infty r^2\rho(r)\diff r<\infty,
\end{equation}
where $\rho$ is the spatial density associated to $\fke$.
\end{lemma}

\noindent
The proof follows the strategy of~\cite{GuReSt22,RaRe13};
we give the details in Appendix~\ref{A:STST}.

An important quantity associated to the steady state is the effective potential $\Psi$ defined in~\eqref{E:EFFPOT} whose properties we analyse next. 

\begin{lemma}\phantomsection\label{lem:effpotprop}
	\begin{enumerate}[label=(\alph*)]
		\item\label{it:effpot1} There exists a unique radius~$\rL>0$ such that 
		\begin{equation}\label{eq:Emindef}
			\min_{]0,\infty[}\Psi=\Psi(\rL)\eqqcolon\Emin<0.
		\end{equation} 
		This radius is given as the unique zero of~$\Psi'$ on~$]0,\infty[$ and it holds that~$\Psi'<0$ on~$]0,\rL[$ and $\Psi'>0$ on $]\rL,\infty[$.
		\item\label{it:effpot2} Let
		\begin{equation}\label{eq:Aepsdef}
			\AEL\coloneqq]\Emin,0[
		\end{equation}
		denote the set of all {\em admissible particle energies}.
		Then, for any $E\in\AEL$ there exist two unique radii $r_\pm(E)$ satisfying
		\begin{equation*}
			0<r_-(E)<\rL<r_+(E)<\infty
		\end{equation*} 
		and
		\begin{equation}\label{eq:rpmdef}
			\Psi\left(r_\pm(E)\right)=E.
		\end{equation}
	\end{enumerate}
\end{lemma}
\begin{proof}
  The assertions follow from the asymptotic behavior of $\Psi$
  and $\Psi'$ at $r=0,\infty$, and the fact that
  $r^3 \Psi'$ is strictly increasing.
\end{proof}

In particular, since~\eqref{eq:rhog} implies that $\rho(r)>0$ is equivalent to $\Psi(r)<E_0$ for $r>0$, we conclude that
\begin{equation}\label{eq:rhoepssupp}
	\supp(\rho)=[r_-(E_0),r_+(E_0)]\eqqcolon[\Rmin,\Rmax]\subset[\Rmin^0,\Rmax^0].
\end{equation}
The steady state has the following regularity properties.

\begin{lemma}\label{lem:ststregularity}
  It holds that
  $U\in C^3([0,\infty[)$ and $\rho\in C^1([0,\infty[)$.
  In addition,
  $U,\rho\in C^\infty([0,\infty[\setminus\{\Rmin,\Rmax\})$.
\end{lemma}
\begin{proof}
	The continuous differentiability of~$\rho$ on~$[0,\infty[$ follows by~\eqref{eq:rhog} since $E_0-U=y\in C^1([0,\infty[)$ and~$g\in C^1(\R)$. Twice differentiating~\eqref{eq:Veps} then yields $U\in C^3([0,\infty[)$. Moreover, observe that $g\in C^\infty(\R\setminus\{0\})$ and that $E_0-\Psi(r)=0$ is equivalent to $r\in\{\Rmin,\Rmax\}$ by Lemma~\ref{lem:effpotprop}. Thus, we conclude that~$U$ and~$\rho$ are indeed infinitely differentiable on $[0,\infty[\setminus\{\Rmin,\Rmax\}$ by iterating the above argument.
\end{proof}

We note that 
a larger polytropic exponent~$k$ 
leads to higher regularity of~$U$ and~$\rho$.


\subsection{Particle motions and the period function}\label{ssc:particlemotions}


Let~$\fke$ be a steady state as given by Lemma~\ref{lem:stst} with associated effective potential~$\Psi$ defined in~\eqref{E:EFFPOT}. Because the particle energy is of the form $E(r,w)=\frac12w^2+\Psi(r)$, the characteristic flow of the steady state is governed by the system
\begin{equation}\label{eq:charsys}
	\dot r=w,\qquad \dot w=-\Psi'(r).
\end{equation}
Due to the structure of the effective potential established in Lemma~\ref{lem:effpotprop}, the behaviour of solutions of this system is similar to the three-dimensional case~\cite[P.~624f.]{HaReSt22}: The particle energy~$E$ is conserved along solutions of~\eqref{eq:charsys} and every solution with negative energy~$E<0$ is trapped, global in time, and either constant (with energy $E=\Emin$) or time-periodic with the period function $T(E)$ given by~\eqref{eq:Tepsdef}. 

For $E\in\AEL$ let $(R,W)(\cdot,E)\colon\R\to]0,\infty[\times\R$ denote the global solution of~\eqref{eq:charsys} satisfying the initial condition
\begin{equation*}
	R(0,E)=r_-(E),\qquad W(0,E)=0.
\end{equation*}
We further define
\begin{equation}\label{eq:repsdef}
	r(\theta,E)\coloneqq R(\theta\,T(E),E),\quad w(\theta,E)\coloneqq W(\theta\,T(E),E),\qquad E\in\AEL,\,\theta\in\mathbb S^1,
\end{equation}
and note that $(r,w)(\cdot,E)$ is periodic with period~$1$ for $E\in\AEL$.
The period function and the characteristics enjoy the following regularity properties.
\begin{lemma}\label{L:PERIODREGULARITY}
  It holds that
  $(R,W)\in C^2(\R\times\AEL)$ and $T\in C^2(\AEL)$.
\end{lemma}
\begin{proof}
	Since $\Psi\in C^3(]0,\infty[)$ by Lemma~\ref{lem:ststregularity}, the implicit function theorem implies that $r_\pm\in C^3(\AEL)$. We thus conclude the claimed regularity of $(R,W)$ by basic ODE theory.
	
	                Lebesgue's dominated convergence theorem yields that $T$ is continuous on~$\AEL$, cf.~\cite[Lemma~B.7]{HaReSt22}. Because the period function is given as the solution of $W(T(E),E)=0$ with $\dot W(T(E),E)>0$ for $E\in\AEL$, applying the implicit function theorem similarly
to~\cite[Theorem~3.6 et seq.]{MK} then implies that $T\in C^2(\AEL)$. 
\end{proof}

A fundamental ingredient in our analysis is the use of action-angle variables introduced in~\eqref{eq:Iepsdef}--\eqref{thetadef}.
For functions $f\colon\Omega\setminus\{(\rL,0)\}\to\R$ we write
\begin{equation*}
	f(\theta,E)=f((r,w)(\theta,E))
\end{equation*} 
for $(\theta,E)\in\mathbb S^1\times I$. 
Note that integrals change via
\begin{equation}\label{cov_acang}
	\diff w\diff r = T(E)\diff\theta\diff E.
\end{equation}
Action-angle coordinates are not defined at $(r,w)=(\rL,0)\in\Omega$ since the characteristic system~\eqref{eq:charsys} possesses a stationary solution associated to the minimal energy~$\Emin$ there (this corresponds to the so-called elliptic point of the Hamiltonian). The next result controls the behaviour of the action-angle coordinates at this singularity.
Before we proceed, we introduce the abbreviations 
\begin{equation}\label{eq:Tminmaxdef}
	\Tmin\coloneqq\inf_{I}T,\qquad\Tmax\coloneqq\sup_{I}T,
	\end{equation}
and also let
$	\Tpmin\coloneqq\inf_{I}T'$, $\Tpmax\coloneqq\sup_{I}T'$,
$	\Tppmin\coloneqq\inf_{I}T''$, and $\Tppmax\coloneqq\sup_{I}T''$.
We later verify that each of these values is finite, cf.\ Remark~\ref{rem:Tmaxfinite}. 

\begin{lemma}\label{L:REG}
Let $r\colon \mathbb S^1\times\AEL\to]0,\infty[$ be defined as in~\eqref{eq:repsdef}. Then $r\in C^2(\mathbb S^1\times\AEL)$ and there exists a constant $C>0$ such that 
\begin{equation}\label{eq:rthetaestimate}
	|r(\theta,E)-\rL|+|\partial_\theta r(\theta,E)|\leq C{\sqrt{E-\Emin}}
\end{equation}
as well as
\begin{equation*}
	|\partial_Er(\theta,E)|\leq\frac C{\sqrt{E-\Emin}},\qquad (\theta,E) \in \mathbb S^1\times I.
\end{equation*}
The constant~$C$ is bounded in terms of
$\Tmax$, $\Tpmax$, $|I|$,
$\|\Psi''\|_{L^\infty([\Rmin,\Rmax])}$, $\|\Psi'''\|_{L^\infty([\Rmin,\Rmax])}$,  and $\Psi''(\rL)^{-1}$.
\end{lemma}
\begin{proof}
	The claimed regularity of $r$ follows by Lemma~\ref{L:PERIODREGULARITY}.
	For $\Emin\leq E<0$ let $z=z(\cdot,E)\colon\R\to\R$ be the unique global solution of
	\begin{equation}\label{eq:ODEz}
		\ddot z=-\Psi''(R(\cdot,E))\,z,\qquad z(0)=1,\;\dot z(0)=0,
	\end{equation} 
	where we set $R(\cdot,\Emin)\equiv\rL$.
	Gr\"onwall's inequality implies that there exists a constant $C>0$ as described in the statement of the lemma such that
		$|z(s,E)|\leq C$ for $s\in[0,\Tmax],\,E\in I$.
	Furthermore, basic ODE theory yields
	\begin{equation*}
		\partial_ER(s,E)=\partial_Er_-(E)\,z(s,E),\qquad s\in\R,\,E\in I.
	\end{equation*}
	Because 
		$\partial_Er(\theta,E)=\dot R(\theta\,T(E),E)\,T'(E)+\partial_ER(\theta\,T(E),E)$
	for $(\theta,E)\in\mathbb S^1\times I$ and
	\begin{equation}\label{eq:dotRestimate}
		|\dot R(s,E)|=\sqrt{2E-2\Psi(R(s,E))}\leq\sqrt2\,\sqrt{E-\Emin}\leq\frac{\sqrt2\,|I|}{\sqrt{E-\Emin}}
	\end{equation}
	for $(s,E)\in\R\times I$, it remains to show that 
	\begin{equation}\label{eq:rminestimate}
		|\partial_Er_-(E)|\leq\frac C{\sqrt{E-\Emin}},\qquad E\in I,
	\end{equation}
	for some constant $C>0$ as specified in the statement of the lemma. In particular,~\eqref{eq:rthetaestimate} follows by~\eqref{eq:dotRestimate}.
	In order to establish~\eqref{eq:rminestimate}, first observe that
	\begin{equation}\label{eq:rminpartialEformula}
		\partial_Er_-(E)=\frac1{\Psi'(r_-(E))}
	\end{equation}
	for $E\in I$ by the implicit function theorem. 
	Moreover, the radial Poisson equation~\eqref{eq:Veps} yields
	\begin{equation}\label{eq:effpotprime2}
		\Psi''(r)=-\frac{2\Psi'(r)}r + \frac L{r^4}+4\pi\rho(r),\qquad r>0.
	\end{equation}
	In particular,
	\begin{equation}\label{eq:effpotprime2circ}
		\alpha\coloneqq\Psi''(\rL)>0=\Psi'(\rL)
	\end{equation}
	by Lemma~\ref{lem:effpotprop}. This implies that in a small neighbourhood of $E=\Emin$ the denominator in~\eqref{eq:effpotprime2} behaves to the leading order
	like $r_-(E)-r_\ast$, which then easily yields~\eqref{eq:rminestimate} using standard continuity arguments and the mean value theorem.
\end{proof}

\subsection{Limiting behaviour of small steady states}\label{ssc:limit}

For fixed $k>\frac12$ and $\kappa$ satisfying~\eqref{E:SINGLEGAP} we study the behaviour of the steady state family $\fke=\varphi(E)=\k\,\tilde\varphi(E)$ given by Lemma~\ref{lem:stst} as $\k\to0$. In this section, we always add a superscript $\epsilon$ to steady state quantities to make the $\epsilon$-dependencies more visible.

The limiting case~$\epsilon=0$ corresponds to~$U^0\equiv0$. Hence, the associated effective potential is of the form
\begin{equation}\label{eq:effpot0}
	\Psi^0(r)\coloneqq -\frac Mr+\frac L{2r^2},\qquad r>0.
\end{equation} 
The structure of this function is similar as in the case $\epsilon>0$ described in Lemma~\ref{lem:effpotprop}, with
\begin{equation}\label{eq:Emin0def}
	\min_{]0,\infty[}\Psi^0=\Psi^0(\rL^0)=\Emin^0=-\frac{M^2}{2L},\qquad \rL^0=\frac LM,
\end{equation}
and 
\begin{equation*}
	r_\pm^0(E)=\frac{-M\mp\sqrt{M^2+2EL}}{2E}
\end{equation*}
for $E\in\AEL^0$, where
\begin{equation}\label{eq:A0def}
	\AEL^0\coloneqq]\Emin^0,0[.
\end{equation}
Accordingly, the period function takes on the form
\begin{equation}\label{eq:T0def}
	T^0(E)\coloneqq 2\int_{r_-^0(E)}^{r_+^0(E)}\frac{\diff r}{\sqrt{2E-2\Psi^0(r)}}=\frac\pi{\sqrt2}\,\frac M{(-E)^{\frac32}}
\end{equation}
for $E\in\AEL^0$; the latter identity is due to a straight-forward calculation.

\begin{lemma}\phantomsection\label{lem:limit}
	The following assertions hold.
	\begin{enumerate}[label=(\alph*)]
		\item\label{it:limit10} $\Emin^\epsilon\to\Emin^0$ and $E_0^\epsilon\to\kappa$ as $\epsilon\to0$; recall~\eqref{eq:Emindef} and~\eqref{eq:Emin0def}.
		\item\label{it:limit20} $\Tmin^\epsilon\to\Tmin^0$ and $\Tmax^\epsilon\to\Tmax^0$ as $\epsilon\to0$, where the limiting action interval (compare~\eqref{eq:Iepsdef}) is
		\begin{equation}\label{eq:I0def}
			I^0\coloneqq]\Emin^0,\kappa[
		\end{equation}
		and $\Tmin^0$, $\Tmax^0$ are defined similar to~\eqref{eq:Tminmaxdef};
		recall~\eqref{eq:Tepsdef}, and~\eqref{eq:T0def}. Moreover,
		\begin{equation}
			c\leq (T^\epsilon)^{(j)}(E)\leq C,\qquad E\in I^\epsilon.
		\end{equation}
		In particular, $T^\epsilon$ is strictly increasing on~$I^\epsilon$ for $0\leq\epsilon<\epsilon_0$.
	\end{enumerate}
\end{lemma}
\begin{proof}
  The proof of these convergences is rather technical and
  postponed to Appendix~\ref{ssc:convergence}.
	Part~\ref{it:limit10} is shown in Lemmas~\ref{lem:limitrhopot} and~\ref{lem:limitrLEmin},
	part~\ref{it:limit20} is proven in Lemmas~\ref{lem:limitTminmax}, \ref{lem:limitdETminmax} and~\ref{lem:limitdE2Tminmax}. 
\end{proof}

\subsection{Linearisation}\label{S:LINEARISATION}


In order to analyse the linearised operator~$\L$ given by~\eqref{eq:Ldef} with
methods from functional analysis, we first define the transport operator~$\T$ in a weak sense, based on~\cite[Def.~2.1]{ReSt20}:

	For a function~$f\in H$ the transport term $\T f$ {\em exists weakly} if there exists some $\mu\in H$ such that for every test function $g\in C^1_c(\Omega)$,
	\begin{equation*}
		\langle f,\T g\rangle_H=-\langle\mu,g\rangle_H,
	\end{equation*} 
	where $\T g$ is given by~\eqref{eq:transportdef}.
	In this case, $\T f\coloneqq\mu$ {\em weakly}. 
	The domain $\mathrm D(\T)$ of~$\T$ is the subspace of~$H$ where~$\T$ exists weakly, 
	while the domain of the squared transport operator is defined as
	\begin{equation*}
		\mathrm D(\T^2)\coloneqq\{f\in H\mid f\in\mathrm D(\T),\;\T f\in\mathrm D(\T)\}.
	\end{equation*}

We collect the following properties of the transport operator and its square as in~\cite{HaReSt22}, see also~\cite[Prop.~5.1]{GuReSt22} and~\cite{ReSt20}; further properties of~$\T$ can be derived as in these papers.

\begin{lemma}[Properties of~$\T$ and~$\T^2$]\phantomsection\label{lem:transportproperties}
		\begin{enumerate}[label=(\alph*)]
		\item\label{it:Tprop1} $\T\colon\mathrm D(\T)\to H$ is skew-adjoint as a densely defined operator on $H$, i.e., $\T^\ast=-\T$,
                  and $\T^2\colon\mathrm D(\T^2) \to H$ is self-adjoint.
		\item\label{it:Tprop2} The domains of $\T$ and $\T^2$ can be characterised in action-angle coordinates as follows: 
		\begin{align*}
			\mathrm{D}(\T^m) = \Big \{f \in H  \mid\,&  f(\cdot, E) \in H^m_{\theta} \text{ for a.e. } E \in I \\
			& \text{and } \sum_{j=1}^{m}\int_{I} \frac{T(E)^{1-2j}}{|\varphi'(E)|} \int_{\mathbb S^1} |\partial_\theta^j f(\theta, E)|^2 \diff\theta\diff E< \infty\Big\}
		\end{align*} 
		for $m\in\{1,2\}$, where 
		\begin{equation}\label{eq:h1theta}
			H^1_\theta\coloneqq \{ y\in H^1(]0,1[) \, | \, y(0)=y(1)\},\qquad H^2_\theta\coloneqq\{ y\in H^1_{\theta}\mid\dot y \in H^1_\theta\}.
		\end{equation}
		In addition, for $f\in\mathrm D(\T^m)$ with $m\in\{1,2\}$ and a.e.\ $(\theta,E)\in\mathbb S^1\times I$,
		\begin{equation}\label{eq:transportaaweak}
			(\T^m f)(\theta,E) = \left(\frac{1}{T(E)}\right)^m(\partial_\theta^mf)(\theta,E).
		\end{equation}
		\item\label{it:Tprop3} The kernel of $\T$ consists of functions only depending on~$E$, i.e., 
		\begin{equation}\label{eq:kernelTransport}
			\ker(\T) = \left\{f \in H \mid \exists g\colon \R \to \R \text{ s.t. } f(r,w) = g(E(r,w))\text{ a.e.\ on } \Omega\right\}.
		\end{equation}
		\item\label{it:Tprop6} $\T$ reverses $w$-parity
		and the restricted operator $\T^2\big|_{\H}\colon\mathrm D(\T^2)\cap\H\to\H$ is self-adjoint.
		\item\label{it:Tprop8} The spectrum and the essential spectrum of $-\T^2$ are of the form
		\begin{equation*}
			\sigma(-\T^2) = \sigmaess(-\T^2)=\overline{\left(\frac{2\pi\N_0}{T(I)}\right)^2},\qquad\sigma(-\T^2\big|_\H) = \sigmaess(-\T^2\big|_\H)=\overline{\left(\frac{2\pi\N}{T(I)}\right)^2}.
		\end{equation*}
	\end{enumerate}
\end{lemma}
\begin{proof}
	The skew-adjointness of $\T$ can be shown as in~\cite[Thm.~2.2]{ReSt20}, which then yields~\ref{it:Tprop1} using von Neumann's theorem~\cite[Thm.~X.25]{ReSi2}.
	Part~\ref{it:Tprop2} follows similarly to~\cite[Lemma~5.2 and Cor.~5.4]{HaReSt22}.
	The identity~\eqref{eq:transportaaweak} then implies~\ref{it:Tprop3}, while~\ref{it:Tprop6} is evident from parity considerations. 
	Part~\ref{it:Tprop8} is due the observation that $\mathbb S^1\times I\ni(\theta,E)\mapsto\sin(2\pi j\theta)\,\delta_{E^\ast}(E)$ defines an eigendistribution for $-\T^2$ or $-\T^2\big|_\H$;
$j\in\N_0$ or $j\in\N$, respectively. The claimed structures of the spectra
follow by applying Weyl's criterion~\cite[Thm.~7.2]{HiSi} similarly to~\cite[Thm.~5.7]{HaReSt22}.
\end{proof}


We next analyse the response operator~$\Ri$ defined in~\eqref{eq:responsedef}.

\begin{lemma}[Properties of~$\Ri$]\label{lem:responseproperties}
	The linear operator $\Ri\colon H\to H$ is bounded, symmetric, and non-negative (in the sense of quadratic forms, i.e., $\langle\Ri f,f\rangle_H\geq0$ for $f\in H$). 
	The operator
	\begin{equation}\label{eq:sqrtRdef}
		\sqrt\Ri\colon H\to H,\;\sqrt\Ri f(r,w)\coloneqq 2\pi^{\frac32}\,\lv\varphi'(E)\rv\,\frac w{r^2\sqrt{\rho(r)}}\int_{\R}\tilde w\,f(r,\tilde w)\diff\tilde w
	\end{equation}
	is bounded, symmetric, non-negative, and
        $\sqrt\Ri\sqrt\Ri=\Ri$ on~$H$.
	Moreover, $\sqrt\Ri f\in\H$ and $\Ri f\in\H$ for $f\in H$.
\end{lemma}
\begin{proof}
	The claimed statements regarding $\Ri$ follow as in~\cite[Lemma~4.3]{HaReSt22}.
	The properties of $\sqrt\Ri$ can be derived similarly using the important identity
	\begin{equation}\label{eq:ibp}
		\int_{\R}w^2\,\lv\varphi'(E)\rv\diff w=\frac{r^2}\pi\,\rho(r),\qquad r>0.\qedhere
	\end{equation}
\end{proof}

The response operator has a natural connection to the gravitational potential of the linear perturbation.
Similar to~\cite[Sc.~A.1]{HaReSt22}, we thus analyse the properties of such potentials
defined by~\eqref{eq:Ufprime}.

\begin{lemma}\label{lem:potential}
  Let $g\in \mathrm D(\T)$ and $f\coloneqq \T g\in\im(\T)$.
  Then $U_f\in H^2\cap C^1([0,\infty[)$ with
      \begin{equation}\label{eq:Ufestimate}
	\|U_f\|_{H^2}+\|U_f\|_{L^\infty}+\|U_f'\|_{L^\infty}\leq C\|f\|_H
      \end{equation}
      for some constant $C>0$ which can be estimated by~$\epsilon$ and~$k$.
      Furthermore, 
      \begin{equation}\label{eq:Ufprimetransport}
	U_f'(r)=\frac{4\pi^2}{r^2}\int_{\R}w\,g(r,w)\diff w,\qquad r>0,
      \end{equation}
      $\supp(U_f')\subset[\Rmin,\Rmax]$,
      and $U_f(|\cdot|)\in H^2\cap C^1(\R^3)$.
      In action-angle coordinates,
      $U_f\in C^1(\mathbb S^1\times \AEL)$ with
      \begin{equation}\label{E:UEREG}
	|\partial_EU_f|\leq\frac C{\sqrt{E-\Emin}}\, |\partial_rU_f|
      \end{equation}
      on $\mathbb S^1\times I$ for~$C>0$ as in Lemma~\ref{L:REG};
      here $U_f(\theta,E)=U_f(r(\theta,E))$ for
      $(\theta,E)\in\mathbb S^1\times\AEL$. 
	For $j\in\Z$ it holds that $\hat U_f(j,\cdot)\in C^1(\AEL)$ with $\partial_E\hat U_f(j,\cdot)=\widehat{\partial_EU_f}(j,\cdot)$ on $\AEL$.
\end{lemma}
\begin{proof}
	First observe that
	\begin{align*}
		\|\rho_f\|_{L^2(]0,\infty[)}^2&=\pi^2\int_0^\infty\frac1{r^4}\left(\int_\R f(r,w)\diff w\right)^2\diff r\leq\\&\leq C\int_0^\infty\left(\int_\R|\varphi'(E)|\diff w\right)\,\left(\int_\R\frac1{|\varphi'(E)|}\,f(r,w)^2\diff w\right)\diff r\leq C\|f\|_H^2
	\end{align*}
	since $\supp(\rho_f)\subset[\Rmin^0,\Rmax^0]$. In the last step we used the estimate
	\begin{align*}
		\int_{\R}|\varphi'(E)|\diff w= C\,(E_0-\Psi(r))_+^{k-\frac12}\leq C,
	\end{align*}
	which follows by a calculation similar to~\eqref{eq:rhocalculation}. 
	Hence the compact support of $\rho_f$ implies that
        $\rho_f\in L^1\cap L^2(]0,\infty[)$.
	By Lemma~\ref{lem:transportproperties},
	\begin{equation*}
		\frac1\pi\,\int_0^\infty r^2\rho_f(r)\diff r=\langle|\varphi'(E)|,\T g\rangle_H=-\langle\T|\varphi'(E)|,g\rangle_H=0 .
	\end{equation*}	
	In particular, $U_f'(r)=0$ for $r\in[0,\infty[\setminus[\Rmin^0,\Rmax^0]$,
        and $U_f(r)=0$ for $r\geq\Rmax^0$. 
	Thus, $U_f\in C^1([0,\infty[)$ and $U_f,U_f'\in L^1\cap L^\infty([0,\infty[)$
        with
	$
		\|U_f\|_\infty+\|U_f'\|_\infty\leq C\|\rho_f\|_2
	$
	by~\eqref{eq:Ufprime}. 
	Together with the radial Poisson equation we conclude $U_f\in H^2([0,\infty[)$ and the estimate~\eqref{eq:Ufestimate}.
	The identity~\eqref{eq:Ufprimetransport} follows via integration by parts together with a suitable approximation argument similar to~\cite[Eqn.~(A.2)]{HaReSt22}.
	This regularity and Lemma~\ref{L:PERIODREGULARITY} further imply
        that $U_f\in C^1(\mathbb S^1\times\AEL)$. 
	The estimate~\eqref{E:UEREG} hence follows by Lemma~\ref{L:REG}.
\end{proof}
The identity~\eqref{eq:Ufprimetransport} implies that
for $f\in \mathrm D(\T)$ and a.e.\ $(r,w)\in\Omega$,
\begin{equation}\label{eq:responsepotential}
	\Ri f(r,w)=\lv \varphi'(E)\rv \,w\,U_{\T f}'(r),\qquad \sqrt\Ri f(r,w)=\lv \varphi'(E)\rv \,w\,\frac{U_{\T f}'(r)}{\sqrt{2\pi\rho(r)}}.
\end{equation}
The natural domain of definition for the linearised operator $\L=-\T^2-\Ri$ is 
\begin{equation*}
	\mathrm D(\L)\coloneqq\mathrm D(\T^2)\cap\H;
\end{equation*}
recall that~$\L$ governs the dynamics of the odd-in-$w$ part of the linearised perturbation.
We obtain the following properties of this operator:

\begin{lemma}[Properties of~$\L$]\phantomsection\label{lem:Lproperties}
	\begin{enumerate}[label=(\alph*)]
	\item\label{it:Lprop1} The operator $\L\colon\mathrm D(\L)\to\H$ is self-adjoint as a densely defined operator on $\H$.
	\item\label{it:Lprop2} The operators~$\sqrt\Ri$ and~$\Ri$ are relatively $(-\T^2)$-compact~\cite[Def.~14.1]{HiSi} and
	\begin{equation}\label{eq:Lsigmaess}
		\sigmaess(\L)=\sigmaess(-\T^2\big|_{\H})=\overline{\left(\frac{2\pi\N}{T(I)}\right)^2}.
	\end{equation}
	\item\label{it:Lprop3} There exists $c>0$ such that for all $f\in\mathrm D(\L)$,
	\begin{equation}\label{eq:Antonovbound}
		\langle \L f,f\rangle_H\geq c\left(\|f\|_H^2+\|\T f\|_H^2\right).
	\end{equation}
	In particular, the kernel of~$\L$ is trivial and $\sigma(\L)\subset]0,\infty[$.
	\end{enumerate}
\end{lemma}
An estimate of the form \eqref{eq:Antonovbound}
is typically called an {\em Antonov coercivity bound}.
\begin{proof}
	The self-adjointness of~$\L$ is due to the Kato-Rellich theorem~\cite[Thm.~X.12]{ReSi2} and Lemmas~\ref{lem:transportproperties} and~\ref{lem:responseproperties}.
	For part~\ref{it:Lprop2} it suffices to show that
	\begin{equation*}
		\sqrt\Ri\colon\left(\mathrm D(\L),\,\|\T^2\cdot\|_H+\|\cdot\|_H\right)\to\H
	\end{equation*}
	is compact, cf.~\cite[III~Ex.~2.18.(1)]{EnNa}. This can be achieved similarly to~\cite[Thm.~5.9]{HaReSt22} using Lemma~\ref{lem:transportproperties}~\ref{it:Tprop8}, the identity~\eqref{eq:responsepotential}, the bounds from Lemma~\ref{lem:potential}, the compact embedding $H^2([0,\Rmax])\Subset H^1([0,\Rmax])$, and~\eqref{eq:ibp}.
	
	For the last part we first recall the classical~\cite{An1961} Antonov coercivity bound
	\begin{equation}\label{eq:Antonovclassical}
		\langle\L f,f\rangle_H\geq\int_{\Omega}\frac1{\lv\varphi'(E)\rv}\,\frac{m(r)}{r^3}\,\lv f(r,w)\rv^2\diff(r,w)
	\end{equation}
	for $f\in C^2_c(\Omega)$ odd in~$w$, which can be derived as in~\cite[Lemma~1.1]{GuRe2007} or~\cite[(4.6)]{LeMeRa11}. 
	Extending the estimate~\eqref{eq:Antonovclassical} to $f\in\mathrm D(\L)$ via a standard approximation argument~\cite[Prop.~2]{ReSt20} implies $\sigma(\L)\subset[0,\infty[$ and $\ker(\L)=\{0\}$, cf.~\cite[Cor.~7.2 \&~7.3]{HaReSt22}.
	In order to establish the coercivity bound~\eqref{eq:Antonovbound}, we then proceed as in~\cite[Prop.~7.4]{HaReSt22} and deduce that 
	\begin{equation*}
    	  {\tilde \lambda} \coloneqq\inf_{\substack{f\in\mathrm D(\T)\\f\notin\ker(\T)}}\frac{\langle \L f,f\rangle_H}{\|\T f\|_H^2}=
          \inf_{\substack{f\in\mathrm D(\T)\\f\notin\ker(\T)}}
          \left(1-\frac{\int_0^\infty r^2U_{\T f}'(r)^2\diff r}{4\pi^2\,\|\T f\|_H^2}\right) >0
	\end{equation*}
	using Lemmas~\ref{lem:potential} and~\ref{lem:transportproperties}.
	Combining the latter estimate with Lemma~\ref{lem:transportproperties}~\ref{it:Tprop8} similar to~\cite[Thm.~7.5]{HaReSt22} then concludes the proof of part~\ref{it:Lprop3}.
\end{proof}


\begin{corollary}[Single gap structure]\label{C:NOBANDS}
There exists an $\k_0=\k_0(k)>0$ such that for any $0<\k<\k_0$, the linearised operator associated to~$\fke$ satisfies
\begin{equation*}
\sigmaess(\L) = [\frac{4\pi^2}{(\Tmax)^2},\infty[.
\end{equation*}
\end{corollary}


\begin{proof}
The single gap condition~\eqref{E:SINGLEGAP} is equivalent to $\frac{\Tmax^0}{\Tmin^0}>2$ by~\eqref{eq:Tminmax0}, which by Lemma~\ref{lem:limit} and~\eqref{eq:Lsigmaess} implies the claim.
\end{proof}

\section{Absence of embedded eigenvalues}\label{S:ABSENCE}


For fixed 
$k>\frac12$ we consider the steady states
$\fke$ constructed in Lemma~\ref{lem:stst} with $0<\epsilon<\epsilon_0$,
where $\epsilon_0>0$ be such that the statement of Corollary~\ref{C:NOBANDS}
and the uniform estimates from Lemma~\ref{lem:limit}~\ref{it:limit20} hold.
Further $\epsilon$-independent bounds on, e.g., $\Rmin$, $\Rmax$, $E_0$, and
$\Emin$ for $0<\epsilon<\epsilon_0$ follow by Lemmas~\ref{lem:stst}
and~\ref{lem:limit} after suitably shrinking $\epsilon_0>0$.




The central statement of this section is Theorem~\ref{T:NOEEV}, which states that under the (regularity) assumption $k>1$ there are no embedded eigenvalues of~$\L$, i.e., no eigenvalues inside~$\sigmaess(\L)$ given by Lemma~\ref{lem:Lproperties}. We shall prove this by contradiction. To that end, we first make a simple observation relating the eigenvalues of $\L$ to those of $\tilde{\L}$; the latter operator is obviously well-defined on the domain $\mathrm D(\tilde\L)\coloneqq\mathrm D(\T)$, recall~\eqref{E:FOOP}.

\begin{lemma}\label{L:STEP0}
	Assume that the operator $\L\colon \mathrm D(\L)\to\H$ has an embedded eigenvalue $\frac{4\pi^2m^2}{T(E_m)^2}$ for some $m\in\N$ and $E_m\in\barI$
	with an eigenfunction $h\in \mathrm D(\L)$. Then the function $f=h+\frac{T(E_m)}{2\pi i m}\T h$ enjoys the regularity $f\in\mathrm D(\T)$ and satisfies the identity
	\begin{align}\label{E:EEV3NEW}
		\hat f(\ell,E) = - T_m  \frac{|\varphi'(E)| \widehat{U_f}(\ell,E)}{T_m -\frac{m}{\ell} T(E)},\quad\text{ for a.e.\ } E\in I, \ \ell\in\Zast,
	\end{align}
	where we have introduced the shorthand
	\begin{equation}\label{E:TKDEF0}
		T_m\coloneqq T(E_m).
	\end{equation}
	In addition, the statements of Lemma~\ref{lem:potential} apply to~$U_f$, and
	$\nabla U_f\not\equiv0$.
\end{lemma}
\noindent
Here we employ the convention that a complex-valued function lies in some by
definition real-valued function space like $\mathrm D(\T)$, if its real and
imaginary parts do.

\begin{proof}
	Assume that $\l^2$ with $\l\in\R$ is an eigenvalue of $\L$ with an associated eigenfunction $h\in \mathrm D(\L)=\mathrm D(\T^2)\cap\H$. 
	By Lemma~\ref{lem:Lproperties} we have $\l\neq0$. 
	Using~\eqref{eq:responsepotential} and $U_h=0$ (as $h$ is odd in~$w$), 
	it is then easy to check that the pair of functions
	$f = h \pm\frac{1}{i\l} \T h$ 
	are eigenfunctions of the operator $\Lf$ associated to eigenvalues $\pm i\l$. 
	Observe here that $h\in \mathrm D(\L)$ implies $h \pm\frac{1}{i\l} \T h \in \mathrm D(\T)=\mathrm D(\tilde{\L})$. 
	Therefore, $\frac{2\pi i m}{T_m}$
		is an eigenvalue of $\tilde{\L}$, and using action-angle coordinates 
	and~\eqref{eq:transportaaweak} yields the identity
	\begin{equation}\label{E:EEV}
		\frac1{T(E)}\pa_\theta \left(f + |\varphi'(E)| U_f\right) = \frac{2\pi i m}{T_m} f, \quad\text{ a.e.\ on }\mathbb S^1\times I,
	\end{equation}
	where we recall~\eqref{E:TKDEF0}.
	Since $U_f=\frac{T_m}{2\pi i m}U_{\T h}\in C^1(\mathbb S^1\times \AEL)$ by Lemma~\ref{lem:potential}, we may apply the Fourier transform w.r.t.~$\theta\in\mathbb S^1$ to~\eqref{E:EEV} to obtain the relation 
	\begin{equation}\label{E:EEV1}
		\frac{\ell}{T(E)}\left(\hat f(\ell,E)+|\varphi'(E)| \widehat{U_f}(\ell,E) \right) = \frac{m}{T_m} \hat f(\ell,E), \quad\text{ for a.e.\ }  E\in I, \ \ell\in\mathbb Z,
	\end{equation}
	where we recall \eqref{eq:Fouriercoeffdef}.
	It is convenient to rewrite~\eqref{E:EEV1} in the following form
	\begin{equation}\label{E:EEV2}
		\left(T_m-\frac{m}{\ell} T(E)\right) \hat f(\ell,E) = - T_m |\varphi'(E)| \widehat{U_f}(\ell,E), \quad\text{ for a.e.\ } E\in I, \ \ell\in\Zast.
	\end{equation}
	The strict monotonicity of $I\ni E\mapsto T(E)$ implies that for any given $\ell\in\Zast$, there exists at most one energy $E_\ell\in\barI$ such that $T_m-\frac{m}{\ell} T(E_\ell)=0$. We hence conclude~\eqref{E:EEV3NEW}.
	Lastly, assume that $\nabla U_f\equiv0$.
	Then $U_f\equiv0$ since it decays to $0$ as $r\to\infty$ and thus $f\equiv 0$ a.e.\
	by~\eqref{E:EEV2}.
	By definition of $f$ it follows that $\T h = -\frac{2\pi i m}{T_m}h$ which is impossible
	since $h\neq 0$ is odd in $w$ and $\T$ reverses $w$-parity.
\end{proof}


\begin{remark}\label{R:DEGEN}
	If $T_m-\frac{m}{\ell} T(E_\ell)=0$ for $E_\ell=\Emin\in\barI$, it follows by~\eqref{E:EEV2} that $\widehat{U_f}(\ell,\Emin)=0$; we always extend~$T$ smoothly on~$\barI$ using Remark~\ref{rem:Tmaxfinite}. However, by Lemma~\ref{lem:potential},
	$\widehat{U_f}(\ell,\cdot)$ is only $C^{0,\frac12}$ at $E=\Emin$ and therefore
	\begin{equation*}
		\frac{ \widehat{U_f}(\ell,E)}{T_m -\frac{m}{\ell} T(E)}
		\approx (E-\Emin)^{-\frac12}\ \mbox{as}\ E\to\Emin.
	\end{equation*}
	In particular, the relation~\eqref{E:EEV3NEW} does not make sense pointwise at
	$E=\Emin$, but it does weakly, or more precisely in $L^{2-\nu}(I)$ for any $0<\nu\le1$.
\end{remark}







The previous lemma suggests that the frequence-energy pairs where the denominator on the right-hand side of~\eqref{E:EEV3NEW} vanishes
play a distinguished role in the study of embedded eigenvalues.
%
%
%
%
%
%
%
%
%
The next lemma provides simple quantitative bounds on the range of frequencies
that are nearly resonant.

\begin{lemma}[$\delta$-resonant set]\label{L:RESONANTSIZE0}
	Let $(m,E_m)\in\N\times\barI$ be such that $\frac{4\pi^2m^2}{T_m^2}$ is an eigenvalue of~$\L$.
	Let $0<\delta<\frac12\Tmin$ be given. Consider the {\em $\delta$-resonant set}
	\begin{align*}
		\Ldk \coloneqq \left\{\ell\in\Zast\,\big|\, \exists E\in\barI \ \text{ such that } \ |T_m -\frac{m}{\ell}T(E)|<\delta\right\}.
	\end{align*}
	Then $\Ldk\subset\N$ and there exists a constant $\Cr=C(T_{\text{max}}, T_{\text{min}})>0$ such that 
	\begin{align}
		\lv \frac{\ell}m\rv  + \lv \frac{m}{\ell} \rv &\le \Cr, \ \ \ell \in \Ldk . \label{E:UPPERLOWER}
	\end{align}
\end{lemma}


\begin{proof}
	If there exists an $E\in\barI$ such that $-\delta< T_m -\frac{m}{\ell}T(E) <\delta$ then clearly $\frac{m}{\ell}>0$ (since $\delta<\frac12\Tmin$) and
	\begin{equation*}
		\frac12 \frac{\Tmin}{\Tmax} <\frac{\Tmin-\delta}{\Tmax}\leq\frac{T_m-\delta}{T(E)}<\frac{m}{\ell}< \frac{T_m+\delta}{T(E)}\leq\frac{\Tmax+\delta}{\Tmin}< \frac32 \frac{\Tmax}{\Tmin},
	\end{equation*}
	which implies the claim.
\end{proof}



{\bf Decomposition of the $\delta$-resonant set.}
For $(m,E_m)\in\N\times\barI$ and $\delta<\frac12 \Tmin$ fixed, we decompose $\Ldk $ into three disjoint sets
\begin{equation*}
	\Ldk  = \mathcal R_m \cup \mathcal P_m \cup \mathcal N_m,
\end{equation*}
where
\begin{align}
	\mathcal R_m &\coloneqq\left\{\ell\in \Ldk \,\big|\,\exists E\in\barI \ \text{ such that } \ T_m-\frac{m}{\ell}T(E) = 0\right\},  \notag \\
	\mathcal P_m &\coloneqq\left\{\ell\in \Ldk \,\big|\,T_m -\frac{m}{\ell} T(E)>0 \ \text{ for all } \ E\in\barI\right\}, \label{E:POSITIVESETDEF}\\
	\mathcal N_m &\coloneqq\left\{\ell\in \Ldk \,\big|\,T_m -\frac{m}{\ell} T(E)<0 \ \text{ for all } \ E\in\barI\right\}; \notag
\end{align}
recall that~$T$ is continuous on~$\barI$.
We call the frequencies $\ell\in \mathcal R_m$, $\ell\neq m$, {\em resonant frequencies}.
For any such frequency there exists an energy value $E_\ell \in \barI$ at which
the equation \eqref{E:EEV2} degenerates,
and by the monotonicity of $I\ni E\mapsto T(E)$ this energy value is unique.
In particular,
\begin{equation} \label{E:EELL}
	T_m-\frac{m}{\ell} T(E)
	\begin{cases}
		<0, &E\in]E_\ell,E_0], \\
		\ge 0, &E\in[\Emin,E_\ell].
	\end{cases}
\end{equation}
%
%
An important piece of notation for the proof of Theorem~\ref{T:NOEEV} is given in the following definition.


\begin{definition}[The function $\pkl$]\label{D:PKL0}
	Let $\delta<\frac12\Tmin$.
	For any pair $(m,\ell)$ with $m\in\N$ and $\ell\in\Ldk$ let
	\begin{equation}\label{E:PFDEF}
		\pkl(t)\coloneqq -\frac{\ell}{m}\log(T_m-\frac m\ell t)+ C_p, \qquad \text{for } t\in [\Tmin,\Tmax]\text{ with }t<\frac{\ell}{m}T_m.
	\end{equation}
	Here $C_p>0$ is chosen independent of $m$, $\ell$, and $\k$ so that $\pkl\geq0$ on its domain of definition; this is possible by Lemma~\ref{L:RESONANTSIZE0}.
	Obviously, $\pkl$ is an antiderivative of the map $t\mapsto\frac1{T_m-\frac{m}{\ell}t}$, and $\pkl(t)\to\infty$ as $T_m-\frac{m}{\ell}t\searrow0$.
\end{definition}

\begin{theorem}\label{T:NOEEV}
	Let $k>1$. Then there exists an $\k_0>0$ such that for any $0<\k<\k_0$
	the operator $\L$ has no embedded eigenvalues.
\end{theorem}


\begin{proof}
	By way of contradiction, we assume that there is an eigenvalue in the essential spectrum of $\L$, which by Lemma~\ref{lem:Lproperties}
	means that it is of the form $\frac{4\pi^2m^2}{T(E_m)^2}$ for some $m\in\N$ and $E_m\in\barI$. 
	Let $f\in\mathrm D(\T)$ be a function as in Lemma~\ref{L:STEP0}, i.e., the relation~\eqref{E:EEV3NEW} holds, the statements of Lemma~\ref{lem:potential} apply to $U_f$, and $\nabla U_f\not\equiv0$.
	
	{\em Step 1. An energy-type identity.}
	We multiply~\eqref{E:EEV3NEW} by the complex conjugate of $-\widehat{U_f}(\ell,E)$, sum over $\ell\in\Zast$, and integrate against $T(E)\diff E$.
	By Parseval's theorem, the left-hand side equals
	\begin{equation*}
		- \int_I \int_{\mathbb S^1} f\,U_f\,T(E)\diff\theta\diff E = - \iint_\Omega f\,U_f \diff(r,w) = \frac1{16\pi^3} \int_{\R^3} |\nabla U_f|^2\diff x;
	\end{equation*}
	observe that $\hat f(0,\cdot)=0$ by~\eqref{E:EEV1}.
	As a result, we obtain the identity
	\begin{align}\label{E:EEV4}
		\frac1{16\pi^3}\int |\nabla U_f|^2\diff x
		=  T_m  \sum_{\ell\neq0}\int_I \frac{T(E)|\varphi'(E)|}{T_m -\frac{m}{\ell}T(E)} \lv \widehat{U_f}(\ell,E)\rv^2\diff E.
	\end{align}
	%
	%
	Now we let $\delta=\frac14\Tmin$
	so that the conclusions of Lemma~\ref{L:RESONANTSIZE0} apply. By \eqref{E:EEV4},
	\begin{multline*}
		\frac1{16\pi^3} \int |\nabla U_f|^2\diff x =  T_m  \sum_{\ell\neq0\atop \ell\in (\Ldk)^c}\int_I \frac{T(E)|\varphi'(E)|}{T_m -\frac{m}{\ell}T(E)} \lv \widehat{U_f}(\ell,E)\rv^2\diff E \\
		\ \ \ \ 
		+T_m  \left(\sum_{\ell\in \mathcal R_m}+\sum_{\ell\in \mathcal P_m}+\sum_{\ell\in \mathcal N_m}\right)\int_I \frac{T(E)|\varphi'(E)|}{T_m -\frac{m}{\ell}T(E)} \lv \widehat{U_f}(\ell,E)\rv^2\diff E.
	\end{multline*}
	We rearrange the above identity, use the bound $|T_m -\frac{m}{\ell}T(E)|\ge\delta$ for $\ell\in (\Ldk)^c$, and~\eqref{E:EELL}
	to obtain 
	\begin{align}
		&\frac1{16\pi^3}\int |\nabla U_f|^2\diff x  \leq \frac{T_m }{\delta} \sum_{\ell\neq0\atop \ell\in (\Ldk)^c}\int_I T(E)|\varphi'(E)| \lv \widehat{U_f}(\ell,E)\rv^2\diff E\nonumber\\
		 &+  T_m\sum_{\ell\in \mathcal P_m}\int_I \frac{T(E)|\varphi'(E)|}{|T_m -\frac{m}{\ell}T(E)|} \lv \widehat{U_f}(\ell,E)\rv^2\diff E
		+  T_m \sum_{\ell\in \mathcal R_m}\int_{\Emin}^{E_\ell} \frac{T(E)|\varphi'(E)|}{|T_m -\frac{m}{\ell}T(E)|} \lv \widehat{U_f}(\ell,E)\rv^2\diff E.  \label{E:STPT}
	\end{align}
	By Definition~\ref{D:PKL0},
	$\frac1{T'(E)}\pa_E\left(\pkl(T(E))\right)=\frac1{|T_m -\frac{m}{\ell}T(E)|}$ for
	$\ell\in\mathcal P_m$ and $E\in I$ or $\ell\in\mathcal R_m$ and $E<E_\ell$. We use this
	to rewrite the integrals above and then integrate by parts in $E$.
	For $\ell\in\P_m$ this results in
	\begin{multline}
		\int_I \frac{T(E)|\varphi'(E)|}{T'(E)} \pa_E\left(\pkl(T(E))\right) \lv \widehat{U_f}(\ell,E)\rv^2\diff E \\
		= A_{\ell} + B_{\ell}-  \frac{T(E)|\varphi'(E)|}{T'(E)} \pkl(T(E)) |\widehat{U_f}(\ell,E)|^2\Big|_{E=\Emin}  \le A_{\ell} + B_{\ell}, \label{E:NEW1}
	\end{multline}
	where for $\ell\in \P_m$,
	\begin{align}
		A_{\ell} &\coloneqq- 2\int_I \frac{T(E)|\varphi'(E)|}{T'(E)} \pkl(T(E))\,
		\mathrm{Re}\left(\pa_E\widehat{U_f}(\ell,E)\, \overline{\widehat{U_f}(\ell,E)}\right)
		\diff E,\label{E:ALP}\\
		B_{\ell}&\coloneqq- \int_I \pa_E\left(\frac{T(E)|\varphi'(E)|}{T'(E)}\right) \pkl(T(E)) \lv \widehat{U_f}(\ell,E)\rv^2\diff E. \label{E:BLP}
	\end{align}
	In~\eqref{E:NEW1}, we used $k>1$ to see $\varphi'(E_0)=0$ and the regularity of~$\widehat U_f(\ell,\cdot)$ from Lemma~\ref{lem:potential} to infer that the boundary term at $E=E_0$ vanishes.
	Analogously, for $\ell\in\mathcal R_m$ we have
	\begin{multline}
		\int_{\Emin}^{E_\ell} \frac{T(E)|\varphi'(E)|}{T'(E)} \pa_E\left(\pkl(T(E))\right) \lv \widehat{U_f}(\ell,E)\rv^2\diff E\\
		= A_{\ell} + B_{\ell}-  \frac{T(E)|\varphi'(E)|}{T'(E)} \pkl(T(E)) |\widehat{U_f}(\ell,E)|^2\Big|_{E=\Emin}  \le A_{\ell} + B_{\ell}, \label{E:NEWR1}
	\end{multline}
	where for $\ell\in \mathcal R_m$
	\begin{align}
		A_{\ell} &\coloneqq - 2\int_{\Emin}^{E_\ell} \frac{T(E)|\varphi'(E)|}{T'(E)} \pkl(T(E))\,
		\mathrm{Re}\left(\pa_E\widehat{U_f}(\ell,E)\, \overline{\widehat{U_f}(\ell,E)}\right)
		\diff E,\label{E:ALR}\\
		B_{\ell}&\coloneqq- \int_{\Emin}^{E_\ell} \pa_E\left(\frac{T(E)|\varphi'(E)|}{T'(E)}\right) \pkl(T(E)) \lv \widehat{U_f}(\ell,E)\rv^2\diff E. \label{E:BLR}
	\end{align}
	In order to see that the boundary term at $E=E_\ell$ in~\eqref{E:NEWR1} vanishes, first note that we may assume $E_\ell>\Emin$. If, in addition, $E_\ell<E_0$, we have $\widehat U_f(\ell,E_\ell)=0$ by~\eqref{E:EEV2} and thus obtain $|\widehat U_f(\ell,E)|^2\pkl(T(E))\to0$ as $E\nearrow E_\ell$ using the regularities of $\widehat U_f$ and $T$. Otherwise, $E_\ell=E_0$ and the boundary term vanishes because $k>1$.
	
	{\em Step 2. Estimates for $A_\ell$, $\ell\in\mathcal P_m\cup\mathcal R_m$.}
	The main challenge in our estimates is that the term $\pa_E\widehat{U_f}(\ell,E)$ at $E=\Emin$ inherits the singular behaviour $(E-\Emin)^{-\frac12}$ and
	for any given $\ell\in\mathbb Z$ the function $E\mapsto \pa_E\widehat{U_f}(\ell,\cdot)$ just fails to be in $L^2(I)$. To go around this we shall 
	introduce powers of 
	\begin{equation*}
		\delta E: = E- \Emin
	\end{equation*}
	as weights in our estimates.
	For any $\ell\in\mathcal P_m$ we have from~\eqref{E:ALP} and Cauchy-Schwarz
	\begin{align*}
		\lv A_\ell\rv^2 
		& \leq4 \lv \int_I \frac{T(E)|\varphi'(E)|}{T'(E)} \pkl(T(E)) \pa_E\widehat{U_f}(\ell,E)(\delta E)^\frac12 \overline{\widehat{U_f}(\ell,E)}(\delta E)^{-\frac12}\diff E\rv^2 \notag\\
		& \le 4\int_I  \lv \pa_E\widehat{U_f}(\ell,E)\rv^2T(E)\,\delta E\,|\varphi'(E)|\diff E
		\int_I \frac{\pkl^2(T(E))}{T'(E)^2}  \lv\widehat{U_f}(\ell,E)\rv^2
		\frac{T(E)}{\delta E}|\varphi'(E)|\diff E.
	\end{align*}
	Applying Cauchy's inequality and summing over $\ell\in\mathcal P_m$ yields
	\begin{multline}
		\sum_{\ell\in\mathcal P_m} \lv A_\ell\rv
		\leq\sum_{\ell\in\mathcal P_m} \int_I  \lv \pa_E\widehat{U_f}(\ell,E)\rv^2T(E)\,\delta E\,|\varphi'(E)|\diff E\\
		+ \sum_{\ell\in\mathcal P_m} \int_I \frac{\pkl^2(T(E))}{T'(E)^2}
		\lv\widehat{U_f}(\ell,E)\rv^2 \frac{T(E)}{\delta E}|\varphi'(E)|\diff E.
		\label{E:PKBOUND0}
	\end{multline}
	By the same arguments as above we conclude that
	\begin{multline}
		\sum_{\ell\in\mathcal R_m} \lv A_\ell\rv 
		\leq\sum_{\ell\in\mathcal R_m} \int_{\Emin}^{E_\ell}  \lv \pa_E\widehat{U_f}(\ell,E)\rv^2T(E)\,\delta E\,|\varphi'(E)|\diff E\\
		+ \sum_{\ell\in\mathcal R_m} \int_{\Emin}^{E_\ell} \frac{\pkl^2(T(E))}{T'(E)^2}  \lv\widehat{U_f}(\ell,E)\rv^2\frac{T(E)}{\delta E}|\varphi'(E)|\diff E.
		\label{E:RKBOUND0}
	\end{multline}
	%
	%
	The first sums on the right-hand sides of~\eqref{E:PKBOUND0}--\eqref{E:RKBOUND0} respectively combine to give
	\begin{multline}
		\sum_{\ell\in\mathcal P_m} \int_I  \lv \pa_E\widehat{U_f}(\ell,E)\rv^2T(E)\,\delta E\,|\varphi'|\diff E 
		+ \sum_{\ell\in\mathcal R_m} \int_{\Emin}^{E_\ell}  \lv \pa_E\widehat{U_f}(\ell,E)\rv^2T(E)\,\delta E\,|\varphi'|\diff E \\
		\leq\sum_{\ell\in\mathbb Z} \int_I \lv \pa_E\widehat{U_f}(\ell,E)\rv^2T(E)\,\delta E\,|\varphi'|\diff E = \int_{\mathbb S^1\times I} \lv \pa_E U_f\rv^2T(E)\,\delta E\,|\varphi'| \diff (\theta,E), \label{E:PRELIM}
	\end{multline}
	where we have used the Plancherel identity in the last line.
	We use~\eqref{E:UEREG} and change variables $\theta\mapsto r$ keeping in mind that
	$\frac{\pa r}{\pa\theta}=T(E)\sqrt{2E-2\Psi(r)}$
	to obtain
	\begin{align}
		\int_{\mathbb S^1\times I}& \lv \pa_E U_f(\theta,E)\rv^2T(E)\,\delta E\,|\varphi'|\diff(\theta,E) \le C\int_{\mathbb S^1\times I} \lv \pa_r U_f(\theta,E)\rv^2 T(E)\,|\varphi'|
		\diff(\theta,E)\notag\\
		&= C \epsilon \int_{\Rmin}^{\Rmax} \lv \pa_r U_f(r)\rv^2
		\int_{\Psi(r)}^{E_0}(E_0-E)^{k-1}\frac{\diff E}{\sqrt{E-\Psi(r)}}\diff r\notag\\
		&= C \epsilon \int_{\Rmin}^{\Rmax} \lv \pa_r U_f(r)\rv^2 (E_0-\Psi(r))^{k-1/2}\diff r\notag\\
		&\leq C \epsilon \int_{\Rmin}^{\Rmax} \lv \pa_r U_f(r)\rv^2\diff r
		= c_1 \epsilon \|\nabla U_f\|_{L^2(\mathbb R^3)}^2,\label{E:DERBOUNDKEY}
	\end{align}
	where we have applied the standard integral identity
	\begin{equation}\label{eq:integralidentitygamma}
		\int_a^b\left(s-a\right)^\alpha\,\left(b-s\right)^\beta\diff s = \frac{\Gamma(\alpha+1)\,\Gamma(\beta+1)}{\Gamma(\alpha+\beta+2)}\,\left(b-a\right)^{\alpha+\beta+1},\qquad \alpha,\beta>-1,\;a\leq b.
	\end{equation}
	In the estimate~\eqref{E:DERBOUNDKEY} only the general assumption $k>\frac12$ and the uniform-in-$\k$ bounds on
	$\Rmin$, $\Rmax$, $E_0$, and $\Emin$ have been used; constants denoted by $C$ never
	depend on $\epsilon, m$, or $\ell$.

	It remains to estimate the second sums on the right-hand sides
	of~\eqref{E:PKBOUND0} and~\eqref{E:RKBOUND0} respectively. 
	We start with the resonant contribution from~\eqref{E:RKBOUND0}.
	We use the inequality
	$4\pi^2\ell^2|\widehat{U_f}(\ell,E)|^2\le
	\int_{\mathbb S^1}|\pa_\theta U_f(\theta,E)|^2\diff\theta$,
	recall that $\pa_\theta U_f = \pa_r U_f \frac{\pa r}{\pa\theta}$,
	change variables $\theta\mapsto r$,
	and	apply the estimates from Lemma~\ref{lem:limit}
	to find that
	\begin{align}\label{E:RESBOUND}
		&
		\sum_{\ell\in\Ri_m}\int_{\Emin}^{E_\ell} \frac{\pkl^2(T(E))}{T'(E)^2}
		\lv\widehat{U_f}(\ell,E)\rv^2 \frac{T(E)}{\delta E} |\varphi'(E)|\diff E\notag \\
		&
		\quad {}\le \sum_{\ell\in\Ri_m} \frac1{4\pi^2\ell^2}
		\int_{\Emin}^{E_\ell} \int_{\mathbb S^1} \lv\pa_\theta U_f(\theta,E) \rv^2
		\frac{\pkl^2(T(E))}{T'(E)^2}  \frac{T(E)}{\delta E}
		|\varphi'(E)|\diff\theta\diff E \notag\\
		&
		\quad {}\le \sum_{\ell\in\Ri_m} \frac1{2\pi^2\ell^2}
		\int_{r_-(E_\ell)}^{r_+(E_\ell)} \lv \pa_r U_f\rv^2
		\int_{\Psi(r)}^{E_\ell} \frac{\pkl^2(T(E))}{T'(E)^2}
		\frac{T^2(E)}{\delta E} \sqrt{2E-2\Psi(r)}\,|\varphi'(E)|\diff E \diff r \notag\\
		&
		\quad {}\leq C \sum_{\ell\in\Ri_m} \frac{1}{\ell^2}
		\int_{\Rmin}^{\Rmax} \lv \pa_r U_f\rv^2\diff r\,
		\int_{\Emin}^{E_\ell} \pkl^2(T(E))  (\delta E)^{-\frac12} |\varphi'(E)|\diff E \notag\\
		&
		\quad {} = C \sum_{\ell\in\Ri_m} \frac{1}{\ell^2}
		\int_{\Rmin}^{\Rmax} \lv \pa_r U_f\rv^2\diff r\, I_{\ell m}\leq C\k\sum_{\ell\in\Ri_m}\frac{\|\nabla U_f\|_{L^2(\mathbb R^3)}^2}{\ell^2}
		 = c_2\k\|\nabla U_f\|_{L^2(\mathbb R^3)}^2.
	\end{align}
	In order to see that the constant $c_2$ is indeed independent of
	$\k$ and $m$ we must estimate the energy integrals $I_{\ell m}$
	accordingly. First we note that for $\ell\in\Ri_m$ and
	$E\in [\Emin,E_\ell[$,
	\begin{equation}\label{Tediffest}
		T_m - \frac{m}{\ell} T(E) = \frac{m}{\ell} (T(E_\ell)-T(E))\geq C (E_\ell - E),
	\end{equation}
	where we used Lemmas~\ref{lem:limit} and~\ref{L:RESONANTSIZE0}. For some
	$\alpha>0$ sufficiently small we estimate 
	$\pkl(T(E))$ against $C(T_m - \frac{m}{\ell} T(E))^{-\alpha}$ and apply~\eqref{eq:integralidentitygamma} to find that
	\[ 
	\frac{1}{\epsilon}
	I_{\ell m} \leq C \int_{\Emin}^{E_\ell}(E_0 - E)^{k-1} 
	\frac{(E_\ell - E)^{-2\alpha}}{\sqrt{E-\Emin}}\diff E
	\leq C (E_0 - \Emin)^{k-\frac{1}{2}-2\alpha},
	\] 
	which is uniformly bounded as required, provided $\alpha>0$ is sufficiently small;
	recall that $k>1$.
	By a completely analogous argument
	\begin{multline}\label{E:POSBOUND}
		\sum_{\ell\in \mathcal P_m}
		\int_I \frac{\pkl^2(T(E))}{T'(E)^2}  \lv\widehat{U_f}(\ell,E)\rv^2
		\frac{T(E)}{\delta E} |\varphi'(E)|\diff E\\
		\leq C \sum_{\ell\in \mathcal P_m} \frac{1}{\ell^2}
		\int_{\Rmin}^{\Rmax} \lv \pa_r U_f\rv^2\diff r\, I_{\ell m}
		\le  c_3\k\|\nabla U_f\|_{L^2(\mathbb R^3)}^2
	\end{multline}
	for some $\k,m,\ell$-independent constant $c_3>0$.
	The difference to the estimate \eqref{E:RESBOUND} is that for $\ell\in \mathcal P_m$
	the energy integrals $I_{\ell m}$ extend over the whole energy interval $I$,
	and \eqref{Tediffest} is replaced by
	\[
	T_m - \frac{m}{\ell} T(E) > \frac{m}{\ell} (T(E_0)-T(E))\geq C (E_0 - E).
	\]
	Hence
	\[
	\frac{1}{\epsilon}
	I_{\ell m}
	\leq C (E_0 - \Emin)^{k-\frac{1}{2}-2\alpha}
	\int_{0}^{1} (1-s)^{k-1-2\alpha} s^{-\frac{1}{2}} \diff s,
	\]  
	which is again uniformly bounded.

	{\em Step 3. Estimates for $B_\ell$, $\ell\in\mathcal P_m\cup\mathcal R_m$, see~\eqref{E:BLP} and~\eqref{E:BLR}.}
	These estimates are analogous to the bounds~\eqref{E:RESBOUND} and~\eqref{E:POSBOUND},
	and we obtain
	\begin{align}
		&\sum_{\ell\in\mathcal P_m}\lv \int_I \pa_E\left(\frac{T(E)|\varphi'(E)|}{T'(E)}\right) \pkl(T(E)) \lv \widehat{U_f}(\ell,E)\rv^2\diff E\rv \notag\\
		&\quad{} +\sum_{\ell\in\mathcal R_m}\lv \int_{\Emin}^{E_\ell} \pa_E\left(\frac{T(E)|\varphi'(E)|}{T'(E)}\right) \pkl(T(E)) \lv \widehat{U_f}(\ell,E)\rv^2\diff E\rv \
		\le c_4 \k \|\nabla U_f\|_{L^2(\mathbb R^3)}^2\label{E:STEP4FIXL}
	\end{align}
	for some universal constant $c_4>0$.
	Here, we again rely on the assumption $k>1$ to guarantee the integrability of
	$\varphi''$ near $E=E^0$ and use the uniform bounds on~$T$, $T'$, and~$T''$.
	
	{\em Step 4. Conclusion.}
	We use~\eqref{E:RESBOUND},~\eqref{E:POSBOUND},~\eqref{E:STEP4FIXL},~\eqref{E:DERBOUNDKEY}, and~\eqref{E:STPT} to get
	\begin{align}
		\frac1{16\pi^3}\int |\nabla U_f|^2\diff x
		&\le  \frac{T_m }{\delta} \sum_{\ell\neq0\atop \ell\in (L_{\delta}^k)^c}\int_I T(E)|\varphi'(E)| \lv \widehat{U_f}(\ell,E)\rv^2\diff E  + C\k \int |\nabla U_f|^2\diff x \notag\\
		& \le C\k \int |\nabla U_f|^2\diff x, \notag
	\end{align}
	where the second bound follows by the same argument as in~\eqref{E:RESBOUND}. With~$\epsilon_0>0$ small enough this gives the contradiction for $0<\epsilon<\epsilon_0$, recall $\nabla U_f\not\equiv0$.
\end{proof}

\section{Principal gap analysis}\label{S:GAP}


Throughout this section let~$\fke$ be a steady state given by Lemma~\ref{lem:stst} with fixed $k>\frac12$ and $\epsilon>0$.


\subsection{A Birman-Schwinger principle}\label{ssc:bsm}

In order to characterise the presence of eigenvalues of the linearised operator $\L=-\T^2-\Ri\colon\mathrm D(\L)\to\H$ in the {principal gap} $\mathcal G$ defined in~\eqref{eq:principalgapdef},
we provide a criterion
similar to~\cite[Sc.~8]{HaReSt22}, see also~\cite{MK} and~\cite[Sc.~6]{GuReSt22}. 


\begin{lemma}[A Birman-Schwinger principle, cf.~{\cite[Lemmas~8.1--8.3]{HaReSt22}}]\label{lem:BSprinciple}
	For $\lambda\in\mathcal G$ let 
	\begin{equation}\label{eq:BSoperatordef}
		Q_\lambda\coloneqq\sqrt\Ri\,\left(-\T^2-\lambda\right)^{-1}\sqrt\Ri\,\colon\H\to\H.
	\end{equation}
	We refer to $Q_\lambda$ as the {\em Birman-Schwinger operator} associated to~$\L$.
	This operator is linear, bounded, symmetric, non-negative, and compact.
	Furthermore, the linearised operator~$\L$ possessing an eigenvalue in the principal gap~$\mathcal G$ is equivalent to the existence of $\lambda\in\mathcal G$ such that $Q_\lambda$ has an eigenvalue greater or equal than~$1$.
\end{lemma}


\begin{proof}
	The properties of~$Q_\lambda$ for $\lambda\in\mathcal G$ follow by the properties of $-\T^2$ and $\sqrt\Ri$ derived in Lemmas~\ref{lem:transportproperties} and~\ref{lem:responseproperties}. In particular, $Q_\lambda$ being compact is due to $\sqrt\Ri$ being relatively $(-\T^2)$-compact, cf.\ Lemma~\ref{lem:Lproperties}~\ref{it:Lprop2}.
	
	In order to relate the spectra of~$Q_\lambda$ and~$\L$ to each other, we consider the operators
	\begin{equation*}
		\L_\mu\coloneqq-\T^2-\frac1\mu\,\Ri\colon\mathrm D(\L)\to\H,\qquad \mu>0.
	\end{equation*}
	Similar to Lemma~\ref{lem:Lproperties}, the operators $\L_\mu$ are self-adjoint with $\sigmaess(\L_\mu)=\sigmaess(\L)=\sigma(-\T^2\big|_\H)$.
	Furthermore, for $\lambda\in\mathcal G$ and $\mu\geq1$ there holds
	\begin{equation}\label{eq:BSoperatorLmu}
		\lambda\text{ is an eigenvalue of }\L_\mu\quad\Leftrightarrow\quad\mu\text{ is an eigenvalue of }Q_\lambda.
	\end{equation}
	This equivalency is due to the following two observations: If $f\in\mathrm D(\L)\setminus\{0\}$ solves $\L_\mu f=\lambda f$, then $g\coloneqq\sqrt\Ri f\in\H\setminus\{0\}$ satisfies $Q_\lambda g=\mu g$. Conversely, if $g\in\H\setminus\{0\}$ solves $Q_\lambda g=\mu g$, then $f\coloneqq(-\T^2-\lambda)^{-1}\sqrt\Ri g\in\mathrm D(\L)\setminus\{0\}$ defines a solution of $\L_\mu f=\lambda f$.
	Next, we deduce that
	\begin{equation}\label{eq:Lmumonotone}
		\L=\L_1\text{ has an eigenvalue in }\mathcal G\quad\Leftrightarrow\quad\exists\mu\geq1\colon\L_\mu\text{ has an eigenvalue in }\mathcal G
	\end{equation}
	by the non-negativity of~$\Ri$ (cf.\ Lemma~\ref{lem:responseproperties}) and the positivity of~$\L$ with $\sigmaess(\L)\cap\mathcal G=\emptyset$ (cf.\ Lemma~\ref{lem:Lproperties}) together with the min-max principle for operators~\cite[Prop.~5.12]{HiSi}.
	
	Combining~\eqref{eq:BSoperatorLmu} and~\eqref{eq:Lmumonotone} then concludes the proof.
\end{proof}


We note that $Q_\lambda$ slightly differs from the respective operator defined in~\cite[Eq.~(8.1)]{HaReSt22}. 
The benefit of the definition~\eqref{eq:BSoperatordef} is that $Q_\lambda$ is symmetric, which is not the case in~\cite{HaReSt22}. 

When searching for eigenfunctions of $Q_\lambda$ for $\lambda\in\mathcal G$ associated to non-zero eigenvalues, we may restrict ourselves to the space
\begin{multline}
	\im(Q_\lambda)\subset\im(\sqrt\Ri)\subset\{f\in H\mid \exists F=F(r)\colon f(r,w)=\lv\varphi'(E)\rv\,w\,F(r)\}\\
	=\big\{\Omega\ni(r,w)\mapsto\lv\varphi'(E)\rv\,\frac w{r\sqrt{\rho(r)}}\,F(r)\mid F\in L^2([\Rmin,\Rmax])\big\}.\label{eq:BSimage}
\end{multline}
This leads to the following operator which was first introduced by Mathur~\cite{Ma}.


\begin{lemma}[The Mathur operator, cf.~{\cite[Def.~8.5, Prop.~8.6, and Lemma~8.8]{HaReSt22}}]\label{lem:Mathuroperator}
	For $F\in L^2([\Rmin,\Rmax])$ let $f\in\H$ be defined via 
	\begin{equation*}
		f(r,w)=\lv\varphi'(E)\rv\,\frac w{r\sqrt{\rho(r)}}\,F(r)\quad\text{for a.e. }(r,w)\in\Omega.
	\end{equation*}
	Due to~\eqref{eq:BSimage}, for any $\lambda\in\mathcal G$ there exists a unique $G\in L^2([\Rmin,\Rmax])$ such that
	\begin{equation*}
		Q_\lambda f(r,w)=\lv\varphi'(E)\rv\,\frac w{r\sqrt{\rho(r)}}\,G(r)\quad\text{for a.e. }(r,w)\in\Omega.
	\end{equation*}
	The resulting mapping 
	\begin{equation*}
		\M_\lambda\colon L^2([\Rmin,\Rmax])\to L^2([\Rmin,\Rmax]),\; \M_\lambda F\coloneqq G
	\end{equation*}
	is the {\em Mathur operator}.
	This operator is linear, bounded, symmetric, non-negative, and a compact Hilbert-Schmidt operator~\cite[Thm.~VI.22 et seq.]{ReSi1}. 
	We have the representation
	\begin{equation}\label{eq:MathurHS}
		(\M_\lambda F)(r)=\int_{\Rmin}^{\Rmax} K_\lambda(r,s)\,F(s)\diff s,\qquad F\in L^2([\Rmin,\Rmax]),\;r\in[\Rmin,\Rmax],
	\end{equation}
	with integral kernel $K_\lambda\in C([\Rmin,\Rmax]^2)$ given by
	\begin{equation}\label{eq:Mathurkernel}
		K_\lambda(r,s)\coloneqq\frac{8\pi^{\frac32}}{rs}\sum_{j=1}^{\infty}\int_{I(r)\cap I(s)}\frac{\lv\varphi'(E)\rv}{T(E)}\,
		\frac{\sin(2\pi j\,\theta(r,E))\,\sin(2\pi j\,\theta(s,E))}{\frac{4\pi^2}{T(E)^2}j^2-\lambda}\diff E
	\end{equation}
	for $r,s\in [\Rmin,\Rmax]$,
	where $\theta$ is defined in~\eqref{thetadef} and
	\begin{equation*}
		I(r)\coloneqq\{E\in I\mid r_-(E)<r<r_+(E)\},\qquad r>0.
	\end{equation*} 
\end{lemma}


\begin{proof}
	The operator $\M_\lambda$ being bounded, symmetric, non-negative, and compact follows by the respective properties of the Birman-Schwinger operator~$Q_\lambda$ together with the identity~\eqref{eq:ibp}.
	
	In order to verify that the Mathur operator is a Hilbert-Schmidt operator, the key observation is that using action-angle variables (cf.\ Section~\ref{ssc:particlemotions} and Lemma~\ref{lem:transportproperties}~\ref{it:Tprop2}) yields 
	\begin{equation*}
		(-\T^2-\lambda)^{-1}g(\theta,E) = \frac4{T(E)}\sum_{j=1}^\infty\int_{r_-(E)}^{r_+(E)}\frac{g(\theta(r,E),E)\,\sin(2\pi j\,\theta(r,E))}{\sqrt{2E-2\Psi(r)}}\diff r\,\frac{\sin(2\pi j\theta)}{\frac{4\pi^2}{T(E)^2}j^2-\lambda}
	\end{equation*}
	for $\lambda\in\mathcal G$ and $g\in\H$.
	Inserting the definition of~$\sqrt\Ri$ from Lemma~\ref{lem:responseproperties}, it is then straight-forward to verify that the identity~\eqref{eq:MathurHS} holds.
	The continuity of the kernel~$K_\lambda$ follows by the dominated convergence theorem.
\end{proof}


Due to~\eqref{eq:BSimage}, the Mathur operator still contains all the relevant information of the spectrum of the Birman-Schwinger operator. More precisely, for $\lambda\in\mathcal G$, any $\mu>0$ is an eigenvalue of $Q_\lambda\colon\H\to\H$ if and only if it is an eigenvalue of $\M_\lambda\colon L^2([\Rmin,\Rmax])\to L^2([\Rmin,\Rmax])$; this is similar to~\cite[Lemma~8.10]{HaReSt22}.
In addition, the properties of the Mathur operator derived above together with~\cite[Prop.~5.12]{HiSi} and~\cite[Thm.~VI.6]{ReSi1} imply
\begin{equation*}
	\sup(\sigma(\M_\lambda))=\max(\sigma(\M_\lambda))=\|\M_\lambda\|
\end{equation*}
for $\lambda\in\mathcal G$, where $\|\cdot\|$ denotes the operator norm on $L^2([\Rmin,\Rmax])$ given by
\begin{align}
	M_\lambda\coloneqq\|\M_\lambda\|=&\sup\{\|\M_\lambda F\|_{2}\mid F\in L^2([\Rmin,\Rmax]),\;\|F\|_2=1\}\nonumber\\
	=&\sup\{\langle\M_\lambda F,F\rangle_2\mid F\in L^2([\Rmin,\Rmax]),\;\|F\|_2=1\},\qquad\lambda\in\mathcal G.\label{eq:operatornorm_def}
\end{align}
Overall, we arrive at the following criterion for the presence of eigenvalues of~$\L$ in the principal gap~$\mathcal G$ defined in~\eqref{eq:principalgapdef}.


\begin{prop}[cf.~{\cite[Thm.~8.11]{HaReSt22}}]
\label{prop:BSMprinciple}
	The linearised operator~$\L$ possesses an eigenvalue in the principal gap~$\mathcal G$ if and only if there exists a $\lambda\in\mathcal G$ such that $M_\lambda\geq1$.
\end{prop}



\subsection{Absence of eigenvalues in the principal gap for $k>1$}


We now prove the absence of eigenvalues in the principal gap~$\mathcal G$ defined~\eqref{eq:principalgapdef} using the Birman-Schwinger principle derived above. 

\begin{theorem}\label{T:NOEVINGAP}
	Assume that the polytropic exponent satisfies $k>1$. 
	Then there exists $\epsilon_0>0$ such that the linearised operator~$\L$ associated to the equilibrium~$\fke$ has no eigenvalues in the principal gap~$\mathcal G$ for any $0<\epsilon<\epsilon_0$.
\end{theorem}


\begin{proof}
	In order to apply Proposition~\ref{prop:BSMprinciple}, let $\lambda\in\mathcal G$ and $F\in L^2([\Rmin,\Rmax])$ with $\|F\|_2=1$.
	Using the representation~\eqref{eq:MathurHS} of the Mathur operator~$\M_\lambda$ yields
	\begin{equation}\label{eq:MathurQuadraticForm}
		\langle\M_\lambda F,F\rangle_{2}=8\pi^{\frac32}\sum_{j=1}^\infty\int_I\frac{\lv\varphi'(E)\rv}{T(E)}\,\frac1{\frac{4\pi^2}{T(E)^2}j^2-\lambda}\left(\int_{r_-(E)}^{r_+(E)}\sin(2\pi j\theta(r,E))\frac{F(r)}r\diff r\right)^2\diff E.
	\end{equation}
	We next apply the Cauchy-Schwarz inequality together with the bounds on $\Rmin$ and $\Rmax$ from~\eqref{eq:rhoepssupp} to estimate the radial integral; the constant $C>0$ changes from line to line but is always uniform in $\epsilon\in]0,\epsilon_0[$. In addition, we use the bound $\lambda<\frac{4\pi^2}{\Tmax^2}$ 
	and thus arrive at 
	\begin{equation}\label{eq:Mlambdabound}
		M_\lambda\leq C\sum_{j=1}^\infty\int_I\frac{\lv\varphi'(E)\rv}{T(E)}\,\frac1{\frac{4\pi^2}{T(E)^2}j^2-\frac{4\pi^2}{\Tmax^2}}\diff E.
	\end{equation}
	
	For the first summand, recall that there exists $\epsilon_0>0$ such that $T=T^\epsilon$ is increasing on $I=I^\epsilon$ for $0<\epsilon<\epsilon_0$ by Lemma~\ref{lem:limit}~\ref{it:limit20}. Together with the uniform bounds from the latter lemma and the mean value theorem (cf.\ Lemma~\ref{L:PERIODREGULARITY} for the regularity of $T$) we deduce
	\begin{equation*}
		\frac{4\pi^2}{T(E)^2}-\frac{4\pi^2}{\Tmax^2}=\frac{4\pi^2}{T(E)^2}-\frac{4\pi^2}{T(E_0)^2}\geq \frac1C\,(E_0-E)
	\end{equation*}
	for $E\in I$. Therefore,
	\begin{equation}\label{eq:absenceeqn1}
		\int_I\frac{\lv\varphi'(E)\rv}{T(E)}\,\frac1{\frac{4\pi^2}{T(E)^2}-\frac{4\pi^2}{\Tmax^2}}\diff E\leq C\int_I\frac{\lv\varphi'(E)\rv}{E_0-E}\diff E=Ck\epsilon\int_I(E_0-E)^{k-2}\diff E\leq C\epsilon
	\end{equation}
	because $k>1$; recall that $I=I^\epsilon$ is uniformly bounded as $\epsilon\to0$ by Lemma~\ref{lem:limit}~\ref{it:limit10}.
	
	In order the remaining summands on the right hand side of~\eqref{eq:Mlambdabound}, observe that $\frac{4\pi^2 }{T(E)^2}j^2-\frac{4\pi^2}{\Tmax^2} \geq\frac{j^2}C$ for $j\geq2$ and $E\in I$. Thus,
	\begin{equation}\label{eq:absenceeqn2}
		\sum_{j=2}^\infty\int_I\frac{\lv\varphi'(E)\rv}{T(E)}\,\frac1{\frac{4\pi^2}{T(E)^2}j^2-\frac{4\pi^2}{\Tmax^2}}\diff E\leq C\int_I\lv\varphi'(E)\rv\diff E\leq C\epsilon.
	\end{equation}
	
	Inserting~\eqref{eq:absenceeqn1} and~\eqref{eq:absenceeqn2} into~\eqref{eq:Mlambdabound} implies $M_\lambda\leq C\epsilon$. Applying the Birman-Schwinger-Mathur criterion from Proposition~\ref{prop:BSMprinciple} then concludes the proof.
\end{proof}



\subsection{Existence of pure oscillations for $\frac12<k\le1$}


We now apply Proposition~\ref{prop:BSMprinciple} to prove the existence of pure oscillations \`a la~\cite[Thm.~8.13]{HaReSt22}.

\begin{theorem}\label{T:BSYES}
	Assume that the polytropic exponent satisfies $\frac12<k\leq1$. 
	Then there exists $\epsilon_0>0$ such that the linearised operator~$\L$ associated to the equilibrium~$\fke$ possesses an eigenvalue in the principal gap~$\mathcal G$ for any $0<\epsilon<\epsilon_0$.
\end{theorem}

\begin{proof}
	For $\lambda\in\mathcal G$ and $F\in L^2([\Rmin,\Rmax])$ we rewrite $\langle\M_\lambda F,F\rangle_{2}$ as in~\eqref{eq:MathurQuadraticForm} to deduce
	\begin{equation*}
		\langle\M_\lambda F,F\rangle_2\geq8\pi^{\frac32}\int_I\frac{\lv\varphi'(E)\rv}{T(E)}\,\frac1{\frac{4\pi^2}{T(E)^2}-\lambda}\left(\int_{r_-(E)}^{r_+(E)}\sin(2\pi\theta(r,E))\,\frac{F(r)}r\diff r\right)^2\diff E.
	\end{equation*}
	Now choose $\eta>0$ and a non-empty set $S\subset]\Rmin,\Rmax[$ such that for all $E\in]E_0-\eta,E_0[$ it holds that $S\subset]r_-(E),r_+(E)[$ and $\sin(2\pi\theta(r,E))\geq\frac12$ for $r\in S$; this is possible since $r_\pm$ are smooth and $\theta(\cdot,E)\colon]r_-(E),r_+(E)[\to]0,\frac12[$ is one-to-one for $E\in I$.
	Setting $F\coloneqq\mathds1_S$ leads to
	\begin{align*}
		\limsup_{\lambda\to\frac{4\pi^2}{\Tmax^2}}M_\lambda&\geq\limsup_{\lambda\to\frac{4\pi^2}{\Tmax^2}}\frac{\langle\M_\lambda F,F\rangle_{2}}{\|F\|_2^2}\geq\frac{2\pi^{\frac32}\lv S\rv}{\Rmin^2}\,\limsup_{\lambda\to\frac{4\pi^2}{\Tmax^2}}\int_{E_0-\eta}^{E_0}\frac{\lv\varphi'(E)\rv}{T(E)}\,\frac1{\frac{4\pi^2}{T(E)^2}-\lambda}\diff E\notag\\
		&=C\lv S\rv\int_{E_0-\eta}^{E_0}\frac{\lv\varphi'(E)\rv}{T(E)}\,\frac1{\frac{4\pi^2}{T(E)^2}-\frac{4\pi^2}{\Tmax^2}}\diff E
	\end{align*}
	for some constant $C=C^\epsilon>0$.
	Using the bounds from Lemma~\ref{lem:limit}~\ref{it:limit20} implies
	\begin{equation*}
		\frac{4\pi^2}{T(E)^2}-\frac{4\pi^2}{\Tmax^2}=\frac{4\pi^2}{T(E)^2}-\frac{4\pi^2}{T(E_0)^2}\leq C\left(T(E_0)-T(E)\right)\leq C(E_0-E)
	\end{equation*}
	for $E\in I$.
	Hence,
	\begin{equation*}
		\limsup_{\lambda\to\frac{4\pi^2}{\Tmax^2}}M_\lambda\geq C\lv S\rv\int_{E_0-\eta}^{E_0}\frac{\lv\varphi'(E)\rv}{E_0-E}\diff E=C\lv S\rv\epsilon k\int_{E_0-\eta}^{E_0}(E_0-E)^{k-2}\diff E.
	\end{equation*}
	Because $k\leq1$, the integral in the latter expression is infinite. Applying Proposition~\ref{prop:BSMprinciple} then concludes the proof.
\end{proof}


%



\section{Proof of the main theorem}\label{S:RAGE}



We can now complete the proof of Theorem~\ref{T:MAINTHEOREM}. 
Part~\ref{it:main1} is the content of Theorem~\ref{T:BSYES}. To prove part~\ref{it:main2} we first observe that 
since $\sigma(\L)\subset]0,\infty[=\mathcal G\cup\sigmaess(\L)$ by Corollary~\ref{C:NOBANDS}, 
Theorems~\ref{T:NOEEV} and~\ref{T:NOEVINGAP} imply that there are no eigenvalues
in the spectrum and therefore the pure point spectrum is empty. It remains to show  the damping formula~\eqref{E:LANDAU}.  
To that end we view the linear evolution~\eqref{E:LVP2} as a first order system of the form
\begin{equation}\label{eq:LVPfirstorder}
	\pa_t \psi =A\psi, \quad \psi = \begin{pmatrix} f \\ \partial_t f \end{pmatrix}, \quad A = \begin{pmatrix} 0 & 1 \\ -\L &0 \end{pmatrix}.
\end{equation}
Following~\cite[Sc.~VI.3]{EnNa}, we consider this system on the Hilbert space $\X\coloneqq\left(\mathrm D(\T)\cap\H\right)\times\H$ with
\begin{equation}\label{eq:Xinnerproduct}
	\langle (f,g),(F,G)\rangle_\X\coloneqq\langle\T f,\T F\rangle_H-\frac1{4\pi^2}\int_0^\infty r^2\,U_{\T f}'(r)\,U_{\T F}'(r)\diff r+\langle g,G\rangle_H
\end{equation}
for $(f,g),(F,G)\in\X$. Here we recall~\eqref{eq:normHdef}. If additionally $f\in\mathrm D(\L)$, the above expression can be rewritten as
\begin{equation*}
	\langle (f,g),(F,G)\rangle_\X=\langle\L f,F\rangle_H+\langle g,G\rangle_H
\end{equation*}
using~\eqref{eq:responsepotential}.
Hence, extending Antonov's coercivity bound from Lemma~\ref{lem:Lproperties}~\ref{it:Lprop3} onto $\mathrm D(\T)\cap\H$ via a standard approximation argument~\cite[Prop.~2]{ReSt20} shows that~\eqref{eq:Xinnerproduct} indeed defines an inner product on~$\X$.

The natural domain of definition for the operator~$A$ is $\mathrm D(A)\coloneqq\mathrm D(\L)\times\left(\mathrm D(\T)\cap\H\right)$, which is a dense subset of~$\X$.
Moreover, since $\L\colon\mathrm D(\L)\to\H$ is self-adjoint and invertible by Lemma~\ref{lem:Lproperties}, it is straight-forward to verify that $A\colon\mathrm D(A)\to\X$ is skew-adjoint, i.e., $A^\ast=-A$.
By Stone's theorem~\cite[Thm.~II.3.24]{EnNa}, $A$ thus generates a unitary $C^0$-group and the system~\eqref{eq:LVPfirstorder} is well-posed, i.e., any initial datum $(f_0,g_0)\in\mathrm D(A)$ launches a unique, global solution of the form $\R\ni t\mapsto e^{tA}(f_0,g_0)\in\mathrm D(A)$, cf.~\cite[Thm.~II.6.7]{EnNa}.
The analogous statement clearly carries over to the second order equation~\eqref{E:LVP2}.


	Consider the operator
	\begin{equation*}
		K\colon\X\to\X,\quad K\begin{pmatrix}f\\g\end{pmatrix}\coloneqq\begin{pmatrix}0\\\lv\varphi'(E)\rv U_{\T f}\end{pmatrix},
	\end{equation*}
	which is bounded by Lemma~\ref{lem:potential}.
	Moreover, for any bounded sequence $(f_n,g_n)\subset\X$ we obtain that $(\T f_n)\subset\H$ is bounded by Lemma~\ref{lem:Lproperties}~\ref{it:Lprop3}. Similar to Lemma~\ref{lem:Lproperties}~\ref{it:Lprop2}, we thus conclude that~$K$ is compact by applying Lemma~\ref{lem:potential}.
	Thus, since the point spectrum of~$A$ is empty by the above discussion, the RAGE theorem~\cite{CyFrKiSi} implies
		\begin{equation}
		0=\lim_{T\to\infty}\frac1T\int_0^T \|K e^{tA}(f_0,g_0)\|_\X^2\diff t
		= \lim_{T\to\infty}\frac1T\int_0^T\||\varphi'(E)|U_{\T f(t)}\|_H^2\diff t.\label{E:DECAYAVERAGE}
	\end{equation}
	Furthermore,
	\begin{equation*}
		\frac1{16\pi^3}\|\nabla U_{\T f(t)}\|_{L^2}^2=-\int_{\Omega}U_{\T f(t)}(r)\,\T f(t,r,w)\diff(r,w)\leq\||\varphi'(E)|U_{\T f(t)}\|_H\,\|\T f(t)\|_H
	\end{equation*}
	by Cauchy-Schwarz and
	\begin{equation}\label{E:TfBOUND}
		\|\T f(t)\|_H^2\leq C\langle\L f(t),f(t)\rangle_H\leq C\|(f(t),\partial_tf(t))\|_\X^2=C\|(f_0,g_0)\|_\X^2
	\end{equation}
	due to Lemma~\ref{lem:Lproperties}~\ref{it:Lprop3} and $(e^{tA})_{t\in\R}$ being unitary. Therefore,
	\begin{align*}
		\frac1T\int_0^T \|\nabla U_{\T f(t)}\|_{L^2}^2 \diff t&\leq\frac CT\int_0^T\||\varphi'(E)|U_{\T f(t)}\|_H\,\|\T f(t)\|_H\diff t\nonumber\\&\leq C\left(\frac1T\int_0^T\||\varphi'(E)|U_{\T f(t)}\|_H^2\diff t\right)^{\frac12}\,\left(\frac1T\int_0^T\|\T f(t)\|_H^2\diff t\right)^{\frac12}
	\end{align*}
	for $T>0$ and~\eqref{E:LANDAU} follows by~\eqref{E:DECAYAVERAGE} and~\eqref{E:TfBOUND}.

\appendix

\section{Steady state theory}\label{sc:limit}


\subsection{Proof of Lemma~\ref{lem:stst}}\label{A:STST}


The ansatz~\eqref{eq:ansatzfeps} indeed yields a stationary solution of the system~\eqref{eq:VlasovVfixedL}--\eqref{eq:rhoVfixedL} provided that~$U$ is the potential associated to~$\fke$, i.e.,
\begin{equation}\label{eq:Veps}
	U'(r)=\frac{4\pi}{r^2}\,\int_0^rs^2\rho(s)\diff s,\quad r>0,\qquad \lim_{r\to\infty}U(r)=0,
\end{equation}
where~$\rho$ is induced by~$\fke$ via~\eqref{eq:rhoVfixedL}, i.e,
\begin{equation}\label{eq:rhoeps}
	\rho(r)=\frac\pi{r^2}\,\int_{\R}\fke(r,w)\diff w.
\end{equation}
In order to get a closed system for~$U$, we insert the ansatz~\eqref{eq:ansatzfeps} into~\eqref{eq:rhoeps} and obtain
\begin{align}
	\rho(r)&=\frac{\sqrt2\pi}{r^2}\,\epsilon\int_{\Psi(r)}^{E_0}(E_0-E)_+^k\,(E-\Psi(r))^{-\frac12}\diff E=\nonumber\\
	&=\sqrt2\pi^{\frac32}\frac{\Gamma(k+1)}{\Gamma(k+\frac32)}\,\frac1{r^2}\,\epsilon\,(E_0-\Psi(r))_+^{k+\frac12}\eqqcolon \frac{c_k}{r^2}\,\epsilon\,(E_0-\Psi(r))_+^{k+\frac12}\label{eq:rhocalculation}
\end{align} 
for $r>0$,
where the effective potential $\Psi$ is given by~\eqref{E:EFFPOT}.
Hence, defining
\begin{equation}\label{eq:defg}
	g(z)\coloneqq c_k\,z_+^{k+\frac12},\qquad z\in\R,
\end{equation}
yields that
\begin{equation}\label{eq:rhog}
	\rho(r)=\frac\epsilon{r^2}\,g(E_0-\Psi(r))=\frac\epsilon{r^2}\,g(E_0-U(r)+\frac Mr-\frac L{2r^2}).
\end{equation}
Observe that $g\in C^1(\R)\cap C^\infty(\R\setminus\{0\})$ since $k>\frac12$.
Inspired by~\cite{RaRe13}, we now consider the quantity
\begin{equation*}
	y\coloneqq E_0-U
\end{equation*}
instead of~$U$.
Then~$y$ solves
\begin{equation}\label{eq:yeqn}
	y'(r)=-\frac{4\pi c_k}{r^2}\,\epsilon\int_0^rg(y(s)+\frac Ms-\frac L{2s^2})\diff s,\qquad r>0.
\end{equation}
We equip this equation with the initial condition
\begin{equation}\label{eq:yinitial}
	y(0)=\kappa
\end{equation}
for prescribed $\kappa$ satisfying the single gap condition~\eqref{E:SINGLEGAP}.
It is straight-forward to verify that there exists a unique solution $y\in C^1([0,\infty[)$ of~\eqref{eq:yeqn}--\eqref{eq:yinitial}, cf.~\cite{RaRe13}. 
This solution possesses a vacuum region at the centre, more precisely, $\rho(r)=0$ and $y(r)=\kappa$ for $0\leq r\leq \Rmin$ with
	$\Rmin=\Rmin^0$ given by~\eqref{eq:Rmindef}
and $\Rmin>0$ is the maximal radius with this property. 
Furthermore, inserting $y\leq\kappa<0$ into~\eqref{eq:rhog} yields  $\rho(r)=0$ for $r\geq\Rmax^0$ where $\Rmax^0$ is given by~\eqref{eq:Rmindef}.
Hence, the limit $y_\infty\coloneqq\lim_{r\to\infty}y(r)\in]-\infty,0[$
exists. Then, setting
	$E_0\coloneqq y_\infty<0$
and
	$U\coloneqq E_0-y$
yields a solution of~\eqref{eq:Veps}.
The estimate for the $w$-part of the support in~\eqref{eq:fepssupport} follows by
	\[
		\frac12\,w^2\leq E_0-E(r,w)=y(r)+\frac Mr-\frac L{2r^2}\leq\frac M{\Rmin}, \ \ (r,w)\in\supp(\fke),
		\]
where we have used the bounds $y\le0$, $r\ge \Rmin$, and $-\frac L{2r^2}<0$ on the galaxy support. \qed	


\subsection{Convergence of the steady state family}\label{ssc:convergence}


The aim of this section is to prove Lemma~\ref{lem:limit}.
For fixed $k>\frac12$ and $\kappa$ as above, this requires a detailed
understanding of the behaviour of the steady states $\fke$ given by Lemma~\ref{lem:stst} as $\epsilon\to0$.
We hence add a superscript $\epsilon$ to all steady state quantities to make the $\epsilon$-dependencies more visible, i.e., $\rho^\epsilon$ is the stationary mass density, $U^\epsilon$ is the stationary potential, and the associated local mass function is given by
\begin{equation}\label{eq:mepsdef}
	m^\epsilon(r)\coloneqq4\pi\int_0^rs^2\rho^\epsilon(s)\diff s,\qquad r\geq0.
\end{equation}

The first result is similar to~\cite[Lemma~3.3]{GuReSt22} and forms the basis for all further convergence results.
Recall Lemma~\ref{lem:ststregularity} for the regularity properties of~$\rho^\epsilon$ and~$U^\epsilon$.

\begin{lemma}\label{lem:limitrhopot}
  As $\epsilon\to0$ it holds that
  $\rho^\epsilon,(\rho^\epsilon)'\to0$ uniformly
  on $[0,\infty[$, $\Ms^\epsilon\to0$, $E_0^\epsilon\to\kappa$,
      and	\begin{equation}\label{eq:Vepslimit}
		(U^\epsilon)^{(j)}\to0 \text{ uniformly on } [0,\infty[,\quad j\in\{0,1,2,3\}.
	\end{equation}
\end{lemma}
\begin{proof}
  Inserting the uniform bound of the radial
  support given by~\eqref{eq:rhoepssupp} and the estimate~$y^\epsilon\leq y^\epsilon(0)=\kappa$ into~\eqref{eq:rhog} yields the uniform convergence $\rho^\epsilon\to0$, which immediately implies that $\Ms^\epsilon\to0$; recall~\eqref{eq:Mepsdef}.
	Together with $(y^\epsilon)'(r)=-\frac{m^\epsilon(r)}{r^2}$ for $r>0$ by~\eqref{eq:yeqn} this also leads to $(y^\epsilon)'\to0$ uniformly.
	After integration, we then deduce that $y^\epsilon\to\kappa$ uniformly on~$[0,\infty[$ and $E_0^\epsilon=y_\infty^\epsilon\to\kappa$.
	    In addition, after
            differentiating~\eqref{eq:rhog} we obtain the uniform convergence $(\rho^\epsilon)'\to0$.
	Combining all these limits then also yields the uniform convergence of the second and third derivative of~$U^\epsilon$ after further differentiating~\eqref{eq:Veps}.
\end{proof}


We next prove that~$\rL^\epsilon$ and~$r_\pm^\epsilon$ converge  to~$\rL^0$ and~$r_\pm^0$, respectively, as $\epsilon\to0$; recall Lemma~\ref{lem:effpotprop} and Section~\ref{ssc:limit} for the definitions of these radii.
We start with the minimising radius~$\rL^\epsilon$ and the associated minimal energy value~$\Emin^\epsilon$.

\begin{lemma}\label{lem:limitrLEmin}
	It holds that $\rL^\epsilon\to\rL^0$ and $\Emin^\epsilon\to\Emin^0$ as $\epsilon\to0$.
\end{lemma}
\begin{proof}
	The radius $\rL^\epsilon$ is given as the unique zero of the increasing function
	\begin{equation*}
		\xi^\epsilon\colon]0,\infty[\to\R,\;\xi^\epsilon(r)\coloneqq M+m^\epsilon(r)-\frac Lr
	\end{equation*}
	for $\epsilon\geq0$; in the case $\epsilon=0$ we have $m^0\equiv0$.
	Because $m^\epsilon\to0$ uniformly on~$[0,\infty[$ as~$\epsilon\to0$ by Lemma~\ref{lem:limitrhopot}, we obtain that $\rL^\epsilon$ indeed tends to $\rL^0$ as $\epsilon\to0$.
	
	Together with the uniform convergence $U^\epsilon\to0$ from Lemma~\ref{lem:limitrhopot} we then deduce that $\Emin^\epsilon=\Psi^\epsilon(\rL^\epsilon)\to\Psi^0(\rL^0)=\Emin^0$ as $\epsilon\to0$.
\end{proof}

The next step is to show analogous properties for the radii $r_\pm^\epsilon(E)$ as well. In addition, we also analyse the behaviour of these radii for $(\k,E)\to(0,\Emin^0)$, i.e., in the near circular regime.
For this sake, let
\begin{equation}\label{eq:Adef}
	\A\coloneqq\{(\epsilon,E)\mid\epsilon\geq0,\;E\in\AEL^\epsilon\}
\end{equation}
denote the set of all {\em admissible $(\epsilon,E)$-pairs}; recall that $\AEL^\epsilon=]\Emin^\epsilon,0[$ by~\eqref{eq:Aepsdef} and~\eqref{eq:A0def}.

\begin{lemma}\phantomsection\label{lem:limitrpm}
	\begin{enumerate}[label=(\alph*)]
		\item\label{it:limitrpm1} The mappings $\A\ni(\epsilon,E)\mapsto r_\pm^\epsilon(E)$ 
		  are continuous at $\epsilon=0$ locally uniformly in~$E$. More precisely, for any $\delta>0$ and $E_1<0$ there exists some $\epsilon_0>0$ such that for all $0\leq\epsilon<\epsilon_0$,
                  $\Emin^\epsilon<E<E_1$, and $\Emin^0<E^\ast<E_1$ with $|E-E^\ast|<\epsilon_0$
                  it holds that $|r_\pm^\epsilon(E)-r_\pm^0(E^\ast)|<\delta$.
		\item\label{it:limitrpm2} The radii~$r_\pm^\epsilon(E)$ converge to~$\rL^0$ as~$E\to\Emin^\epsilon$ and~$\epsilon\to0$. More precisely, for any $\delta>0$ there exist $\epsilon_0>0$ and $\eta>0$ such that for all $0\leq\epsilon<\epsilon_0$ and $\Emin^\epsilon<E<\Emin^\epsilon+\eta<0$
                  it holds that $|r_\pm^\epsilon(E)-\rL^0|<\delta$.
	\end{enumerate}
\end{lemma}
\begin{proof}
	The radius $r_-^\epsilon(E)$ is the unique zero of the decreasing function
	\begin{equation*}
		\xi_E^\epsilon\colon]0,\rL^\epsilon[\to\R,\;\xi_E^\epsilon(r)\coloneqq\Psi^\epsilon(r)-E
	\end{equation*}
	for $\epsilon\geq0$ and $E\in\AEL^\epsilon$. We continuously extend~$r_-^0$ by setting $r_-^0(\Emin^0)\coloneqq\rL^0$ and
	let $\delta>0$ and $E_1\in]\Emin^0,0[$ be arbitrary. By taking $\delta>0$ sufficiently small, we ensure that there exist $E\in]\Emin^0,E_1[$ with $r_-^0(E)+\delta\leq\rL^0$. We then observe that, as $\delta>0$, for such $E$, $\xi_E^0(r_-^0(E)+\delta)$ is uniformly negative, that is,
	%
	\begin{equation*}
		\zeta_-\coloneqq\sup\{\xi_E^0(r_-^0(E)+\delta)\mid\Emin^0<E\leq E_1,\,r_-^0(E)+\delta\leq\rL^0\}<0.
	\end{equation*}
	Now let~$\epsilon$, $E^\ast$, and~$E$ be as specified in the statement of part~\ref{it:limitrpm1} of the lemma for some $\epsilon_0>0$ which we define below.
	If $r_-^0(E^\ast)+2\delta\geq\rL^\epsilon$, then $r_-^\epsilon(E)\leq \rL^\epsilon\leq r_-^0(E^\ast)+2\delta$.
	Otherwise, i.e., $r_-^0(E^\ast)+2\delta<\rL^\epsilon$, we also have $r_-^0(E^\ast)+\delta<\rL^0$ after choosing~$\epsilon_0>0$ sufficiently small by Lemma~\ref{lem:limitrLEmin}. 
	Due to $\zeta_-<0$ and the uniform bound $r_-^0(E^\ast)+\delta\geq r_-^0(E_1)$, we then obtain that $\xi_E^\epsilon(r_-^0(E^\ast)+\delta)<0$ after potentially shrinking $\epsilon_0>0$ again according to Lemma~\ref{lem:limitrhopot}. This implies that $r_-^\epsilon(E)\leq r_-^0(E^\ast)+\delta$.
	Showing that $r_-^\epsilon(E)\geq r_-^0(E^\ast)-\delta$ works similarly. An analogous proof yields the respective estimates for $r_+$ as well, which concludes the proof of part~\ref{it:limitrpm1}.
	
	Part~\ref{it:limitrpm2} then follows by combining part~\ref{it:limitrpm1} with the convergence of~$\Emin^\epsilon$ from Lemma~\ref{lem:limitrLEmin}.
\end{proof}

In particular, Lemmas~\ref{lem:limitrhopot} and~\ref{lem:limitrpm} imply the convergence of the radial support of the steady state, i.e., 
\begin{equation}\label{eq:Rmaxlimit}
	\Rmin^\epsilon=\Rmin^0,\qquad \Rmax^\epsilon\to\Rmax^0\text{ as }\epsilon\to0,
\end{equation}
recall~\eqref{eq:Rmindef} and~\eqref{eq:rhoepssupp}.


We now consider the period function $T^\epsilon\colon\AEL^\epsilon\to]0,\infty[$ for $\epsilon\geq0$ defined in~\eqref{eq:Tepsdef} and~\eqref{eq:T0def}. 
The aim is to establish the uniform convergence of $T^\epsilon$ on the energy support~$\overline{I}^\epsilon$ of the steady state as $\epsilon\to0$; recall~\eqref{eq:Iepsdef} and~\eqref{eq:I0def}. The main difficulty in this task is that the set~$I^\epsilon$ changes in~$\epsilon$, in particular, the minimal energy value $\Emin^\epsilon$ depends on~$\epsilon$. 
We thus first consider the case of this \enquote{near circular regime}, i.e., the region where~$E$ gets close to the minimal energy value~$\Emin^\epsilon$.
It turns out that $T^\epsilon$ is essentially determined by $(\Psi^\epsilon)''(\rL^\epsilon)$ in this regime, which is why we start by establishing the following auxiliary result; recall Lemma~\ref{lem:ststregularity} for the regularity of the effective potential~$\Psi^\epsilon$ for $\epsilon>0$.

\begin{lemma}\label{lem:effpotlimitnearcirc}
	Let $j\in\N_0$. Then $(\Psi^\epsilon)^{(j)}(s)$ converges to $(\Psi^0)^{(j)}(\rL^0)$ as $E\to\Emin^\epsilon$ and $\epsilon\to0$ uniformly in $s\in[r_-^\epsilon(E),r_+^\epsilon(E)]$.
	More precisely, for any $\delta>0$ there exist $\epsilon_0>0$ and $\eta>0$ such that for all $0\leq\epsilon<\epsilon_0$ and $\Emin^\epsilon<E<\Emin^\epsilon+\eta<0$ it holds that $|(\Psi^\epsilon)^{(j)}(s)-(\Psi^0)^{(j)}(\rL^0)|<\delta$ for $s\in[r_-^\epsilon(E),r_+^\epsilon(E)]$.
\end{lemma}
\begin{proof}
	Applying Lemmas~\ref{lem:limitrhopot}--\ref{lem:limitrpm} and~\eqref{eq:Rmaxlimit} implies that there exist $\tilde\delta,\epsilon_0,\eta>0$ such that
	\begin{equation*}
		[r_-^\epsilon(E),r_+^\epsilon(E)]\subset[\Rmin^\epsilon+2\tilde\delta,\Rmax^\epsilon-2\tilde\delta]\subset[\Rmin^0+\tilde\delta,\Rmax^0-\tilde\delta]
	\end{equation*}
	for $\Emin^\epsilon<E<\Emin^\epsilon+\eta$ and $0\leq\epsilon<\epsilon_0$. In particular, $(\Psi^\epsilon)^{(j)}$ exists on $[r_-^\epsilon(E),r_+^\epsilon(E)]$ by Lemma~\ref{lem:ststregularity}; observe that $\Psi^0\in C^\infty(]0,\infty[)$ by~\eqref{eq:effpot0}.
	If $j\leq3$, the statement then follows by Lemma~\ref{lem:limitrhopot}.
	In the case of a larger~$j$ one has to iterate the arguments of Lemma~\ref{lem:limitrhopot}, i.e., further differentiate~\eqref{eq:Veps} and~\eqref{eq:rhog}, to deduce that $(U^\epsilon)^{(j)}\to0$ uniformly on~$[\Rmin^0+\tilde\delta,\Rmax^0-\tilde\delta]$ as~$\epsilon\to0$.
\end{proof}

We then obtain the following behaviour of $T^\epsilon(E)$ as $\epsilon\to0$ for~$E$ in the near circular regime.

\begin{lemma}\label{lem:limitTnearcirc}
	The period function $T^\epsilon(E)$ converges to $T^0(\Emin^0)$ as $E\to\Emin^\epsilon$ and $\epsilon\to0$.
	More precisely, for any $\delta>0$ there exist $\epsilon_0>0$ and $\eta>0$ such that for all $0\leq\epsilon<\epsilon_0$ and $\Emin^\epsilon<E<\Emin^\epsilon+\eta<0$ it holds that $|T^\epsilon(E)-T^0(\Emin^0)|<\delta$.
	
	Here, $T^0(\Emin^0)$ denotes the continuous extension of~$T^0$ onto $\Emin^0$. Due to~\eqref{eq:effpot0}, \eqref{eq:Emin0def}, and~\eqref{eq:T0def}, this value is given by
	\begin{equation}\label{eq:T0nearcircdef}
		T^0(\Emin^0)=2\pi\,\frac{L^{\frac32}}{M^2}=\frac{2\pi}{\sqrt{(\Psi^0)''(\rL^0)}}.
	\end{equation}
\end{lemma}
\begin{proof}
	For any $\epsilon\geq0$ and $E\in\AEL^\epsilon$ we obtain that
	\begin{align}
		T^\epsilon(E)&=2\int_{r_-^\epsilon(E)}^{\rL^\epsilon}\frac{\diff r}{\sqrt{2E-2\Psi^\epsilon(r)}}+    2\int_{\rL^\epsilon}^{r_+^\epsilon(E)}\frac{\diff r}{\sqrt{2E-2\Psi^\epsilon(r)}}\nonumber\\
		&=\int_{\Emin^\epsilon}^E\frac{\sqrt2}{\sqrt{(E-\tilde\eta)\,(\Psi^\epsilon)'(r_-^\epsilon(\tilde\eta))^2}}\diff\tilde\eta +\int_{\Emin^\epsilon}^E\frac{\sqrt2}{\sqrt{(E-\tilde\eta)\,(\Psi^\epsilon)'(r_+^\epsilon(\tilde\eta))^2}}\diff\tilde\eta \notag
	\end{align}
	by changing variables via $\tilde\eta=\Psi^\epsilon(r)$, $r=r_\pm^\epsilon(\tilde\eta)$ in both integrals. We focus on the first integral in the last line of the above calculation; the arguments for the second integral are similar.
	
	Due to the extended mean value theorem, for every $\tilde\eta\in]\Emin^\epsilon,E[$ there exists some $s\in]r_-^\epsilon(\tilde\eta),\rL^\epsilon[\subset]r_-^\epsilon(E),\rL^\epsilon[$ such that
	\begin{equation*}
		\frac{(\Psi^\epsilon)'(r_-^\epsilon(\tilde\eta))^2}{\tilde\eta-\Emin^\epsilon}=\frac{(\Psi^\epsilon)'(r_-^\epsilon(\tilde\eta))^2}{\Psi^\epsilon(r_-^\epsilon(\tilde\eta))-\Psi^\epsilon(\rL^\epsilon)}=2\,(\Psi^\epsilon)''(s);
	\end{equation*}
	recall that $(\Psi^\epsilon)'(\rL^\epsilon)=0$.
	Hence, the integrand of the integral under consideration can be rewritten as
	\begin{equation}\label{eq:Tintegrandemvt}
		\frac{\sqrt 2}{\sqrt{(E-\tilde\eta)\,(\Psi^\epsilon)'(r_-^\epsilon(\tilde\eta))^2}} = \frac1{\sqrt{(E-\tilde\eta)\,(\tilde\eta-\Emin^\epsilon)\,(\Psi^\epsilon)''(s)}}.
	\end{equation}
	Now Lemma~\ref{lem:effpotlimitnearcirc} implies that for any $\delta>0$ there exist $\epsilon_0>0$ and $\eta>0$ such that for all $0\leq\epsilon<\epsilon_0$ and $\Emin^\epsilon<E<\Emin^\epsilon+\eta$ it holds that $|(\Psi^\epsilon)''(s)^{-\frac12}-(\Psi^0)''(\rL^0)^{-\frac12}|<\frac\delta{2\pi}$ for $s\in[r_-^\epsilon(E),r_+^\epsilon(E)]$; note that $(\Psi^0)''(\rL^0)=\frac{M^4}{L^3}>0$.
	Together with~\eqref{eq:T0nearcircdef} we thus conclude the following estimate for the integral under consideration for $0\leq\epsilon<\epsilon_0$ and $\Emin^\epsilon<E<\Emin^\epsilon+\eta$:
	\begin{align*}
		\big|\int_{\Emin^\epsilon}^E&\frac{\sqrt2}{\sqrt{(E-\tilde\eta)\,(\Psi^\epsilon)'(r_-^\epsilon(\tilde\eta))^2}}\diff\tilde\eta-\frac12T^0(\Emin^0)\big|\nonumber\\
		&=\big| \int_{\Emin^\epsilon}^E\frac{\sqrt2}{\sqrt{(E-\tilde\eta)\,(\Psi^\epsilon)'(r_-^\epsilon(\tilde\eta))^2}}\diff\tilde\eta - \frac1{\sqrt{(\Psi^0)''(\rL^0)}}\,\int_{\Emin^\epsilon}^E\frac{\diff\tilde\eta}{\sqrt{(E-\tilde\eta)\,(\tilde\eta-\Emin^\epsilon)}}\big|\\
		&\leq\frac\delta{2\pi}\int_{\Emin^\epsilon}^E\frac{\diff\tilde\eta}{\sqrt{(E-\tilde\eta)(\tilde\eta-\Emin^\epsilon)}}=\frac\delta2.\qedhere
	\end{align*}  
\end{proof}

In order to establish the desired uniform convergence of $T^\epsilon$ on the energy support as $\epsilon\to0$, we next verify the pointwise convergence of $T^\epsilon(E)$.
This is based on the following concavity estimate which originates from~\cite[Lemma~2.1~(iii)]{LeMeRa11}; the proof only uses the elliptic equation~\eqref{eq:Veps}.

\begin{lemma}\label{lem:effpotconcav}
  For any $\epsilon\geq0$, $E\in\AEL^\epsilon$, and $r\in[r_-^\epsilon(E),r_+^\epsilon(E)]$
  the following concavity estimate holds:
	\begin{equation}\label{eq:concavityeffpot}
		E-\Psi^\epsilon(r)\geq L\,\frac{(r_+^\epsilon(E)-r)\,(r-r_-^\epsilon(E))}{2r^2\,r_-^\epsilon(E)\,r_+^\epsilon(E)}.
	\end{equation}
\end{lemma}

The continuity of~$T^\epsilon$ at $\epsilon=0$ now follows similar to~\cite[Lemma~B.7]{HaReSt22}. 

\begin{lemma}\label{lem:Tcontinuous0}
	The mapping $\A\ni(\epsilon,E)\mapsto T^\epsilon(E)$ is continuous at $\epsilon=0$; recall~\eqref{eq:Adef}.
	More precisely, for any $\delta>0$ and $E^\ast\in\AEL^0$ there exists $\epsilon_0>0$ such that for all $0\leq\epsilon<\epsilon_0$ and $E\in\AEL^\epsilon$ with $|E^\ast-E|<\epsilon_0$ it holds that $|T^\epsilon(E)-T^0(E^\ast)|<\delta$.
\end{lemma}
\begin{proof}
	For any $\epsilon\geq0$ and $E\in\AEL^\epsilon$ the affine change of variables $r=r_-^\epsilon(E)+(r_+^\epsilon(E)-r_-^\epsilon(E))\,s$ leads to
	\begin{equation*}
		T^\epsilon(E)=\sqrt2\int_0^1\frac{r_+^\epsilon(E)-r_-^\epsilon(E)}{\sqrt{E-\Psi^\epsilon\left(r_-^\epsilon(E)+(r_+^\epsilon(E)-r_-^\epsilon(E))\,s\right)}}\diff s.
	\end{equation*}
	Lemmas~\ref{lem:limitrhopot} and~\ref{lem:limitrpm} imply the pointwise convergence of the integrand in the integral above as $(\epsilon,E)\to(0,E^\ast)$, the concavity estimate from Lemma~\ref{lem:effpotconcav} shows that the integrand is bounded by an integrable, $E$-independent function. Hence, Lebesgue's dominated convergence theorem implies the desired continuity statement.
\end{proof}

Combining Lemmas~\ref{lem:limitTnearcirc} and~\ref{lem:Tcontinuous0} with standard continuity arguments yields the main result of the present section. 

\begin{lemma}\label{lem:limitTepsuniform}
	The period function $T^\epsilon(E)$ converges to $T^0(E^\ast)$ as $\epsilon\to0$ and $E\to E^\ast$ locally uniformly. 
	More precisely, for any $\delta>0$ and $E_1<0$ there exists $\epsilon_0>0$ such that for all $0\leq\epsilon<\epsilon_0$ as well as $\Emin^0<E^\ast<E_1$ and $\Emin^\epsilon<E<E_1$ with $|E^\ast-E|<\epsilon_0$ it holds that $|T^\epsilon(E)-T^0(E^\ast)|<\delta$.
\end{lemma}

We now conclude the convergence of the minimal and maximal value of the period function on the steady state support.
Recall that $I^\epsilon=]\Emin^\epsilon,E_0^\epsilon[$ for~$\epsilon\geq0$ by~\eqref{eq:Iepsdef} and~\eqref{eq:I0def} as well as the definitions of~$\Tmin^\epsilon$ and~$\Tmax^\epsilon$ in~\eqref{eq:Tminmaxdef}.

\begin{lemma}\label{lem:limitTminmax}
	It holds that $\lim_{\epsilon\to0}\Tmin^\epsilon=\Tmin^0$ and $\lim_{\epsilon\to0}\Tmax^\epsilon=\Tmax^0$, with
	\begin{equation}\label{eq:Tminmax0}
		0<\Tmin^0=T^0(\Emin^0)=2\pi\,\frac{L^{\frac32}}{M^2}<\frac\pi{\sqrt 2}\,\frac M{(-\kappa)^\frac32}=T^0(\kappa)=\Tmax^0<\infty
	\end{equation}
	due to~\eqref{E:SINGLEGAP}, \eqref{eq:T0def}, and~\eqref{eq:T0nearcircdef}.
\end{lemma}
\begin{proof}
	Combine Lemma~\ref{lem:limitTepsuniform} with the limit results $\Emin^\epsilon\to\Emin^0$ and $E_0^\epsilon\to\kappa$ as $\epsilon\to0$ established in Lemmas~\ref{lem:limitrhopot} and~\ref{lem:limitrLEmin}.
\end{proof}

We later show that~$T^\epsilon$ is increasing on~$I^\epsilon$ for $0\leq\epsilon\ll 1$, which implies that~$\Tmin^\epsilon$ and~$\Tmax^\epsilon$ are attained on the boundary of~$I^\epsilon$.


We now establish results similar to Lemmas~\ref{lem:limitTepsuniform} and~\ref{lem:limitTminmax} for the first and second order derivatives of the period function; recall the regularity of~$T^\epsilon$ shown in Lemma~\ref{lem:ststregularity}.
We first derive suitable representations of these derivatives. Proceeding as in Lemma~\ref{lem:ststregularity} leads to a relation between $\partial_ET^\epsilon(E)$ and $(\partial_EW^\epsilon)(T^\epsilon(E),E)$, cf.~\cite[Lemma~A.12]{MK}. However, as $\partial_EW^\epsilon$ is only implicitly known as the solution of a suitable ODE, this quantity is rather hard to control; in particular in the vicinity of the minimum of the potential well $\Emin$.

Instead, we proceed as suggested in~\cite[Sc.~B.3]{HaReSt22} and derive a suitable integral expression for~$\partial_ET^\epsilon$.

\begin{lemma}\label{lem:limitdETformula}
	For $\epsilon\geq0$ and $E\in\AEL^\epsilon$ it holds that
	\begin{equation}\label{eq:dETformula}
		(T^\epsilon)'(E)=\frac1{E-\Emin^\epsilon}\,\int_{r_-^\epsilon(E)}^{r_+^\epsilon(E)}\frac{G_0^\epsilon(r)}{\sqrt{2E-2\Psi^\epsilon(r)}}\diff r,
	\end{equation}
	where the continuous function $G_0^\epsilon\colon]0,\infty[\to\R$ is defined by
	\begin{equation}\label{eq:G0def}
		G_0^\epsilon(r)\coloneqq\begin{cases}
			\frac{(\Psi^\epsilon)'(r)^2-2(\Psi^\epsilon(r)-\Emin^\epsilon)\,(\Psi^\epsilon)''(r)}{(\Psi^\epsilon)'(r)^2},&r\neq\rL^\epsilon,\\
			0,&r=\rL^\epsilon.
		\end{cases}
	\end{equation}
\end{lemma}
\begin{proof}
	Taylor expanding~\eqref{eq:G0def} in the limit $r\to\rL^\epsilon$ yields
	\begin{equation}\label{eq:G0taylor}
		G_0^\epsilon(r)=\frac{-\frac13\,(r-\rL^\epsilon)^3\,(\Psi^\epsilon)''(\rL^\epsilon)\,(\Psi^\epsilon)'''(\rL^\epsilon)+o((r-\rL^\epsilon)^3)}{(r-\rL^\epsilon)^2\,(\Psi^\epsilon)''(\rL^\epsilon)^2 + o((r-\rL^\epsilon)^2)}\to 0\quad\text{as }r\to\rL^\epsilon
	\end{equation}
	by~\eqref{eq:effpotprime2circ}.
	Thus, $G_0^\epsilon$ is indeed continuous and, in particular, the integral on the right-hand side of~\eqref{eq:dETformula} is well-defined.
	A  calculation similar to~\cite[Thm.~2.1]{ChWa86} then yields the identity~\eqref{eq:dETformula}.
\end{proof}

It is again crucial to understand the behaviour of~$(T^\epsilon)'(E)$ in the near circular regime, i.e., when~$E$ gets close to the minimal energy value~$\Emin^\epsilon$. One difficulty in this task is the factor in front of the integral on the right-hand side of~\eqref{eq:dETformula}. It is hence convenient to rewrite the integral expression~\eqref{eq:dETformula} as follows:
\begin{equation}\label{eq:dETroot}
	(T^\epsilon)'(E)=-\frac1{E-\Emin^\epsilon}\,\int_{r_-^\epsilon(E)}^{r_+^\epsilon(E)}\frac{G_0^\epsilon(r)}{(\Psi^\epsilon)'(r)}\,\partial_r\left[\sqrt{2E-2\Psi^\epsilon(r)}\right]\diff r
\end{equation}
for~$\epsilon\geq0$ and $E\in\AEL^\epsilon$,
with the intention to integrate by parts in~\eqref{eq:dETroot}.
For this sake we introduce the function 
\begin{equation}\label{eq:G1def}
	G_1^\epsilon\colon]0,\infty[\to\R,\qquad G_1^\epsilon(r)\coloneqq\begin{cases}\frac{G_0^\epsilon(r)}{(\Psi^\epsilon)'(r)},&\text{if }r\neq \rL^\epsilon,\\
		-\frac13\frac{(\Psi^\epsilon)'''(\rL^\epsilon)}{(\Psi^\epsilon)''(\rL^\epsilon)^2},&\text{if }r=\rL^\epsilon,
	\end{cases}
\end{equation}
with $G_0^\epsilon$ defined in Lemma~\ref{lem:limitdETformula}; recall that $(\Psi^\epsilon)''(\rL^\epsilon)>0$ by~\eqref{eq:effpotprime2circ}. 
A Taylor expansion similar to~\eqref{eq:G0taylor} yields that $G_1^\epsilon$ is continuous on~$]0,\infty[$. In fact, recalling from~\eqref{eq:G0taylor} that $G_0^\epsilon$ and $(\Psi^\epsilon)'$ are both smooth and $\mathcal O(r-\rL^\epsilon)$ as $r\to\rL^\epsilon$, and that $(\Psi^\epsilon)'$ vanishes only at $\rL^\epsilon$, we deduce that $G_1^\epsilon$ is smooth on $]\Rmin^\epsilon,\Rmax^\epsilon[$ (recall the regularities established in Lemma~\ref{lem:ststregularity}) and that its derivatives admit explicit representation in terms of derivatives of $\Psi^\epsilon$.

For $\epsilon\geq0$ and $E\in\AEL^\epsilon$ we now continue the calculation~\eqref{eq:dETroot} and integrate by parts:
\begin{equation}\label{eq:dETrootnumerator}
	(T^\epsilon)'(E)=\frac1{E-\Emin^\epsilon}\,\int_{r_-^\epsilon(E)}^{r_+^\epsilon(E)}(G_1^\epsilon)'(r)\,\sqrt{2E-2\Psi^\epsilon(r)}\diff r.
\end{equation}
Let us hence analyse the behaviour of integrals of the latter form in the near circular regime.

\begin{lemma}\label{lem:limitsintsnearcirc}
	Let $]0,\infty[\ni r\mapsto F(r)$ be continuous. For fixed $\epsilon\geq0$ it holds that
	\begin{align}
		\lim_{E\searrow\Emin^\epsilon}\int_{r_-^\epsilon(E)}^{r_+^\epsilon(E)}F(r)\,\sqrt{2E-2\Psi^\epsilon(r)}\diff r&=0. \label{eq:intlimitrootnumerator}
	\end{align}
	Moreover, it holds that
	\begin{equation}\label{eq:intderivativeroot}
		\partial_E\left[\int_{r_-^\epsilon(E)}^{r_+^\epsilon(E)}F(r)\,\sqrt{2E-2\Psi^\epsilon(r)}\diff r\right]=\int_{r_-^\epsilon(E)}^{r_+^\epsilon(E)}\frac{F(r)}{\sqrt{2E-2\Psi^\epsilon(r)}}\diff r,\qquad E\in\AEL^\epsilon.
	\end{equation}
\end{lemma} 
\begin{proof}
	First observe that Lemma~\ref{lem:effpotprop} implies that $r_\pm^\epsilon(E)\to\rL^\epsilon$ as $E\searrow\Emin^\epsilon$;
	the proof is similar to the one of Lemma~\ref{lem:limitrpm}. Then~\eqref{eq:intlimitrootnumerator} is obvious.
	The derivative relation~\eqref{eq:intderivativeroot} is straight-forward to verify using Lebesgue's dominated convergence theorem.
\end{proof}

Due to the mean value theorem together with~\eqref{eq:dETrootnumerator}, \eqref{eq:intlimitrootnumerator}, and~\eqref{eq:intderivativeroot}, we conclude that for any $E\in\AEL^\epsilon$ there exists some $\tilde E\in]\Emin^\epsilon,E[$ such that
\begin{equation}\label{eq:dETmeanvalue}
	(T^\epsilon)'(E)=\int_{r_-^\epsilon(\tilde E)}^{r_+^\epsilon(\tilde E)}\frac{(G_1^\epsilon)'(r)}{\sqrt{2\tilde E-2\Psi^\epsilon(r)}}\diff r.
\end{equation}
We hence arrive at the following limiting behaviour of $(T^\epsilon)'$ in the near circular regime:

\begin{lemma}\label{lem:limitdETnearcirc}
	The derivative of the period function $(T^\epsilon)'(E)$ converges to $(T^0)'(\Emin^0)$ as $E\to\Emin^\epsilon$ and $\epsilon\to0$.
	More precisely, for any $\delta>0$ there exist $\epsilon_0>0$ and $\eta>0$ such that for all $0\leq\epsilon<\epsilon_0$ and $\Emin^\epsilon<E<\Emin^\epsilon+\eta<0$ it holds that $|(T^\epsilon)'(E)-(T^0)'(\Emin^0)|<\delta$.
	
	Here, $(T^0)'(\Emin^0)$ denotes the continuous extension of~$(T^0)'$ onto $\Emin^0$. Due to~\eqref{eq:effpot0}, \eqref{eq:Emin0def}, and~\eqref{eq:T0def}, 
	 this value is given by
	\begin{equation}\label{eq:dET0nearcircdef}
		(T^0)'(\Emin^0)=6\pi\,\frac{L^{\frac52}}{M^4}=\pi\,\frac{(G_1^0)'(\rL^0)}{\sqrt{(\Psi^0)''(\rL^0)}}.
	\end{equation}
\end{lemma}
\begin{proof}
	We proceed as in the proof of Lemma~\ref{lem:limitTnearcirc} and change variables via $\tilde\eta=\Psi^\epsilon(r)$ in~\eqref{eq:dETmeanvalue} to deduce that for any $\epsilon\geq0$ and $E\in\AEL^\epsilon$ there exists some $\tilde E\in]\Emin^\epsilon,E[$ such that
	\begin{equation}\label{eq:dETchangeofvariables}
		(T^\epsilon)'(E)=\int_{\Emin^\epsilon}^{\tilde E}\frac{(G_1^\epsilon)'(r_-^\epsilon(\tilde\eta))}{\sqrt{2(\tilde E-\tilde\eta)\,(\Psi^\epsilon)'(r_-^\epsilon(\tilde\eta))^2}}\diff\tilde\eta +\int_{\Emin^\epsilon}^{\tilde E}\frac{(G_1^\epsilon)'(r_+^\epsilon(\tilde\eta))}{\sqrt{2(\tilde E-\tilde\eta)\,(\Psi^\epsilon)'(r_+^\epsilon(\tilde\eta))^2}}\diff\tilde\eta.
	\end{equation}
	Lemma~\ref{lem:effpotlimitnearcirc} implies that for any $\delta>0$ there exist $\epsilon_0>0$ and $\eta>0$ such that for all $0\leq\epsilon<\epsilon_0$ and $\Emin^\epsilon<E<\Emin^\epsilon+\eta$ it holds that $|(G_1^\epsilon)'(\tilde s)\,(\Psi^\epsilon)''(s)^{-\frac12}-(G_1^0)'(\tilde s)\,(\Psi^0)''(\rL^0)^{-\frac12}|<\frac\delta{2\pi}$ for $s,\tilde s\in[r_-^\epsilon(E),r_+^\epsilon(E)]$; recall that $(G_1^\epsilon)'$ admits explicit representation in terms only of derivatives of $\Psi^\epsilon$ derived from~\eqref{eq:G1def}. 
	For $0\leq\epsilon<\epsilon_0$ and $\Emin^\epsilon<E<\Emin^\epsilon+\eta$ we thus conclude the following estimate for the first integral in~\eqref{eq:dETchangeofvariables} after rewriting the integrand with the extended mean value theorem similar to~\eqref{eq:Tintegrandemvt}:
	\begin{align*}
		\big|&\int_{\Emin^\epsilon}^{\tilde E}\frac{(G_1^\epsilon)'(r_-^\epsilon(\tilde\eta))}{\sqrt{2(\tilde E-\tilde\eta)\,(\Psi^\epsilon)'(r_-^\epsilon(\tilde\eta))^2}}\diff\tilde\eta-\frac12(T^0)'(\Emin^0)\big| \nonumber\\
		&=\big| \int_{\Emin^\epsilon}^{\tilde E}\frac{(G_1^\epsilon)'(r_-^\epsilon(\tilde\eta))}{\sqrt{2(\tilde E-\tilde\eta)\,(\Psi^\epsilon)'(r_-^\epsilon(\tilde\eta))^2}}\diff\tilde\eta - \frac{(G_1^0)'(\rL^0)}{2\sqrt{(\Psi^0)''(\rL^0)}}\,\int_{\Emin^\epsilon}^{\tilde E}\frac{\diff\tilde\eta}{\sqrt{(\tilde E-\tilde\eta)\,(\tilde\eta-\Emin^\epsilon)}}\big|\\
		&\leq\frac\delta{2\pi}\int_{\Emin^\epsilon}^{\tilde E}\frac{\diff\tilde\eta}{\sqrt{(\tilde E-\tilde\eta)(\tilde\eta-\Emin^\epsilon)}}=\frac\delta2.
	\end{align*}  
	Similar arguments also apply to the second integral in~\eqref{eq:dETchangeofvariables}.
\end{proof}


The next step is again to verify a suitable pointwise convergence of~$(T^\epsilon)'$ as $\epsilon\to0$.

\begin{lemma}\label{lem:dETcontinuous0}
	The mapping $\A\ni(\epsilon,E)\mapsto (T^\epsilon)'(E)$ is continuous at $\epsilon=0$; recall~\eqref{eq:Adef}.
	
	More precisely, for any $\delta>0$ and $E^\ast\in\AEL^0$ there exists $\epsilon_0>0$ such that for all $0\leq\epsilon<\epsilon_0$ and $E\in\AEL^\epsilon$ with $|E^\ast-E|<\epsilon_0$ it holds that $|(T^\epsilon)'(E)-(T^0)'(E^\ast)|<\delta$.
\end{lemma}
\begin{proof}
	First observe that $G_0^\epsilon\to G_0^0$ as $\epsilon\to0$ locally uniformly by Lemmas~\ref{lem:limitrhopot} and~\ref{lem:limitrLEmin}; recall the definition of~$G_0^\epsilon$ in~\eqref{eq:G0def}.
	Then the claimed continuity follows similarly to Lemma~\ref{lem:Tcontinuous0} using the concavity estimate~\eqref{eq:concavityeffpot} and Lebesgue's dominated convergence theorem applied to the representation~\eqref{eq:dETformula} of $(T^\epsilon)'(E)$; also note that $\Emin^\epsilon\to\Emin^0$ by Lemma~\ref{lem:limitrLEmin}.
\end{proof}

We then arrive at the desired convergence results for the derivative of the period function.

\begin{lemma}\label{lem:limitdETminmax}
	It holds that $\lim_{\epsilon\to0}\Tpmin^\epsilon=\Tpmin^0$ and $\lim_{\epsilon\to0}\Tpmax^\epsilon=\Tpmax^0$, with
	\begin{equation}\label{eq:dETminmax0}
		0<\Tpmin^0=(T^0)'(\Emin^0)=6\pi\,\frac{L^{\frac52}}{M^4}<\frac{3\pi}{2\sqrt 2}\,\frac M{(-\kappa)^\frac52}=(T^0)'(\kappa)=\Tpmax^0<\infty.
	\end{equation}
\end{lemma}
\begin{proof}
	Combining Lemmas~\ref{lem:limitdETnearcirc} and~\ref{lem:dETcontinuous0} yields that $(T^\epsilon)'(E)$ converges to $(T^0)'(E^\ast)$ as $\epsilon\to0$ and $E\to E^\ast$ locally uniformly in~$E^\ast$; see Lemma~\ref{lem:limitTepsuniform} for similar arguments.
	Because $\Emin^\epsilon\to\Emin^0$ and $E_0^\epsilon\to\kappa$ as $\epsilon\to0$ by Lemmas~\ref{lem:limitrhopot} and~\ref{lem:limitrLEmin} we conclude the desired convergence results.
\end{proof}

The next step is to establish a similar result for the second derivative of the period function; recall that $T^\epsilon\in C^2(\AEL^\epsilon)$ by Lemma~\ref{L:PERIODREGULARITY}.
Again differentiating~\eqref{eq:dETrootnumerator} using~\eqref{eq:intderivativeroot} and rearranging the integrand yields
\begin{equation}\label{eq:dE2Tformula0}
	(T^\epsilon)''(E)=\frac1{(E-\Emin^\epsilon)^2}\int_{r_-^\epsilon(E)}^{r_+^\epsilon(E)}(G_1^\epsilon)'(r)\left(\frac{\Psi^\epsilon(r)-\Emin^\epsilon}{\sqrt{2E-2\Psi^\epsilon(r)}}-\frac12\sqrt{2E-2\Psi^\epsilon(r)}\right)\diff r
\end{equation}
for $E\in\AEL^\epsilon$.
We integrate by parts to rewrite the first summand and arrive at
\begin{equation}\label{eq:dE2Tformula}
	(T^\epsilon)''(E)=\frac1{(E-\Emin^\epsilon)^2}\int_{r_-^\epsilon(E)}^{r_+^\epsilon(E)}\partial_r\left[(G_1^\epsilon)'(r)\frac{\Psi^\epsilon(r)-\Emin^\epsilon}{(\Psi^\epsilon)'(r)}-\frac12G_1^\epsilon(r)\right]\sqrt{2E-2\Psi^\epsilon(r)}\diff r
\end{equation}
for $E\in I^\epsilon$.
This is possible because the function
\begin{equation*}
	G_2^\epsilon\colon]\Rmin^\epsilon,\Rmax^\epsilon[\to\R,\;G_2^\epsilon(r)\coloneqq\begin{cases}
		\partial_r\left[(G_1^\epsilon)'(r)\frac{\Psi^\epsilon(r)-\Emin^\epsilon}{(\Psi^\epsilon)'(r)}-\frac12G_1^\epsilon(r)\right],&\text{ if }r\neq\rL^\epsilon,\\
		0,&\text{ if }r=\rL^\epsilon,
	\end{cases}
\end{equation*}
is continuous by Taylor's theorem; recall Lemma~\ref{lem:ststregularity} and that~$G_1^\epsilon$ defined in~\eqref{eq:G1def} is continuously differentiable.

Applying the extended mean value theorem to~\eqref{eq:dE2Tformula} and using~\eqref{eq:intderivativeroot} yields that for any $E\in I^\epsilon$ there exists $\tilde E\in]\Emin^\epsilon,E[$ such that 
\begin{equation}\label{eq:dE2Tformula2}
	(T^\epsilon)''(E)=\frac1{2(\tilde E-\Emin^\epsilon)}\int_{r_-^\epsilon(\tilde E)}^{r_+^\epsilon(\tilde E)}\frac{G_2^\epsilon(r)}{\sqrt{2\tilde E-2\Psi^\epsilon(r)}}\diff r.
\end{equation}
In order to eliminate the factor $(\tilde E-\Emin^\epsilon)^{-1}$ we integrate by parts again. 
Observing that $G_2^\epsilon$ is again $\mathcal O(r-\rL^\epsilon)$ as $r\to\rL^\epsilon$ and smooth on $]\Rmin^\epsilon,\Rmax^\epsilon[$, we deduce that
 the function $G_3^\epsilon\colon]\Rmin^\epsilon,\Rmax^\epsilon[\to\R$ defined by
\begin{equation*}
	G_3^\epsilon(r)\coloneqq\begin{cases}
		\frac{G_2^\epsilon(r)}{(\Psi^\epsilon)'(r)},&\text{ if }r\neq\rL^\epsilon,\\
		-\frac1{10}\frac{(\Psi^\epsilon)^{(5)}(\rL^\epsilon)}{(\Psi^\epsilon)''(\rL^\epsilon)^3}+\frac12\frac{(\Psi^\epsilon)'''(\rL^\epsilon)\,(\Psi^\epsilon)^{(4)}(\rL^\epsilon)}{(\Psi^\epsilon)''(\rL^\epsilon)^4}-\frac49\frac{(\Psi^\epsilon)'''(\rL^\epsilon)^3}{(\Psi^\epsilon)''(\rL^\epsilon)^5},&\text{ if }r=\rL^\epsilon,
	\end{cases}
\end{equation*}
is continuously differentiable with derivative that is explicitly computable in terms of derivatives of $\Psi^\epsilon$. Here the value of $G_3^\epsilon(\rL^\epsilon)$ follows from Taylor expansion in $G_2^\epsilon$.
Hence,
\begin{equation}\label{eq:dE2Tformula3}
	(T^\epsilon)''(E)
	=\frac1{2(\tilde E-\Emin^\epsilon)}\int_{r_-^\epsilon(\tilde E)}^{r_+^\epsilon(\tilde E)}(G_3^\epsilon)'(r)\,\sqrt{2\tilde E-2\Psi^\epsilon(r)}\diff r
\end{equation}
for $E\in I^\epsilon$ and $\tilde E\in]\Emin^\epsilon,E[$ as in~\eqref{eq:dE2Tformula2}. 
The mean value theorem and~\eqref{eq:intderivativeroot} imply the existence of~$\bar E\in]\Emin^\epsilon,\tilde E[\subset]\Emin^\epsilon,E[$ such that 
\begin{equation}\label{eq:dE2Tformula4}
	(T^\epsilon)''(E)=\frac12\,\int_{r_-^\epsilon(\bar E)}^{r_+^\epsilon(\bar E)}\frac{(G_3^\epsilon)'(r)}{\sqrt{2\bar E-2\Psi^\epsilon(r)}}\diff r
\end{equation}
Because this identity is similar to~\eqref{eq:dETmeanvalue}, we deduce the following behaviour of $(T^\epsilon)''(E)$ in the near circular regime.

\begin{lemma}\label{lem:limitdE2Tnearcirc}
	The second order derivative of the period function $(T^\epsilon)''(E)$ converges to $(T^0)''(\Emin^0)$ as $E\to\Emin^\epsilon$ and $\epsilon\to0$.
	More precisely, for any $\delta>0$ there exist $\epsilon_0>0$ and $\eta>0$ such that for all $0\leq\epsilon<\epsilon_0$ and $\Emin^\epsilon<E<\Emin^\epsilon+\eta<0$ it holds that $|(T^\epsilon)''(E)-(T^0)''(\Emin^0)|<\delta$.
	
	Here, $(T^0)''(\Emin^0)$ denotes the continuous extension of~$(T^0)''$ onto $\Emin^0$ given by
	\begin{equation}\label{eq:dE2T0nearcircdef}
		(T^0)''(\Emin^0)=30\pi\,\frac{L^{\frac72}}{M^6}=\frac\pi2\,\frac{(G_3^0)'(\rL^0)}{\sqrt{(\Psi^0)''(\rL^0)}}.
	\end{equation}
\end{lemma}
\begin{proof}
	The statement can be proven similarly to Lemma~\ref{lem:limitdETnearcirc} by using~\eqref{eq:dE2Tformula4} and the properties of~$G_3^\epsilon$ derived above.
\end{proof}
The next step is again to verify a suitable pointwise convergence of~$(T^\epsilon)''$ as $\epsilon\to0$.

\begin{lemma}\label{lem:dE2Tcontinuous0}
	The mapping $\A\ni(\epsilon,E)\mapsto (T^\epsilon)''(E)$ is continuous at $\epsilon=0$.
	
	More precisely, for any $\delta>0$ and $E^\ast\in\AEL^0$ there exists $\epsilon_0>0$ such that for all $0\leq\epsilon<\epsilon_0$ and $E\in\AEL^\epsilon$ with $|E^\ast-E|<\epsilon_0$ it holds that $|(T^\epsilon)''(E)-(T^0)''(E^\ast)|<\delta$.
\end{lemma}
\begin{proof}
	First observe that $(G_1^\epsilon)'\to (G_1^0)'$ as $\epsilon\to0$ locally uniformly by Lemmas~\ref{lem:limitrhopot} and~\ref{lem:limitrLEmin}; recall that  $(G_1^\epsilon)'$ admits explicit representation in terms of derivatives of $\Psi^\epsilon$. 
	Then the claimed continuity follows similarly to Lemmas~\ref{lem:Tcontinuous0} and~\ref{lem:dETcontinuous0} using the concavity estimate~\eqref{eq:concavityeffpot} and Lebesgue's dominated convergence theorem applied to the representation~\eqref{eq:dE2Tformula0} of $(T^\epsilon)''(E)$.
\end{proof}

We then arrive at the desired convergence results for the second order derivative of the period function.

\begin{lemma}\label{lem:limitdE2Tminmax}
	It holds that $\lim_{\epsilon\to0}\Tppmin^\epsilon=\Tppmin^0$ and $\lim_{\epsilon\to0}\Tppmax^\epsilon=\Tppmax^0$, with
	\begin{equation}\label{eq:dE2Tminmax0}
		0<\Tppmin^0=(T^0)''(\Emin^0)=30\pi\,\frac{L^{\frac72}}{M^6}<\frac{15\pi}{4\sqrt 2}\,\frac M{(-\kappa)^\frac52}=(T^0)''(\kappa)=\Tppmax^0<\infty.
	\end{equation}
\end{lemma}
\begin{proof}
	The proof is based on Lemmas~\ref{lem:limitdE2Tnearcirc} and~\ref{lem:dE2Tcontinuous0} and proceeds similarly as the proof of Lemma~\ref{lem:limitdETminmax}.
\end{proof}

\begin{remark}\label{rem:Tmaxfinite}
	Similar arguments as in the proofs of Lemmas~\ref{lem:limitTnearcirc}, \ref{lem:limitdETnearcirc}, and~\ref{lem:limitdE2Tnearcirc} imply that
	\begin{equation*}
		T^\epsilon(E)\to\frac{2\pi}{\sqrt{(\Psi^\epsilon)''(\rL^\epsilon)}},\qquad(T^\epsilon)'(E)\to\pi\,\frac{(G_1^\epsilon)'(\rL^\epsilon)}{\sqrt{(\Psi^\epsilon)''(\rL^\epsilon)}},\qquad(T^\epsilon)''(E)\to\frac\pi2\,\frac{(G_3^\epsilon)'(\rL^\epsilon)}{\sqrt{(\Psi^\epsilon)''(\rL^\epsilon)}}
	\end{equation*}
	as ${E\searrow\Emin^\epsilon}$
	for fixed $\epsilon\geq0$,
	where $(G_1^\epsilon)'(\rL^\epsilon)$ and $(G_3^\epsilon)'(\rL^\epsilon)$ are explicitly computable in terms of derivatives of $\Psi^\epsilon$ at $\rL^\epsilon$. 
	Together with Lemma~\ref{L:PERIODREGULARITY} and~\eqref{eq:effpotprime2circ} we hence conclude that $\Tmax^\epsilon, \Tpmax^\epsilon,\Tppmax^\epsilon<\infty$ as well as $\Tmin^\epsilon>0$ for any $\epsilon\geq0$.
\end{remark}


\end{document}